\documentclass[11pt,leqno]{amsart}

\usepackage{amsmath}
\usepackage{amssymb}
\usepackage{a4wide}
\usepackage{bbm}
\usepackage{graphics}
\usepackage{epsfig}
\usepackage{color}
\usepackage{mathrsfs}
\usepackage[sans]{dsfont}
\usepackage{yhmath}

\newcommand{\Subsection}[1]{\subsection{ #1} ${}^{}$}

\newtheorem{theorem}{Theorem}[section]
\newtheorem{lemma}[theorem]{Lemma}
\newtheorem{proposition}[theorem]{Proposition}

\newcounter{hypo}
\newcounter{hypoa}

\newcounter{hypoaa}
\newcounter{hypobb}

\def\N{{\mathbb N}} 
\def\R{{\mathbb R}}

\def\S{{\mathbb S}}

\def\CE{\mathcal {E}}

\def\CH{\mathcal {H}}

\def\CO{\mathcal {O}}

\def\CV{\mathcal {V}}

\def\one{\mathds{1}}
\def\re{\mathop{\rm Re}\nolimits}
 \def\im{\mathop{\rm Im}\nolimits}

\def\Op{\mathop{\rm Op}\nolimits}

\newcommand{\Id}{\operatorname{Id}}

\newcommand{\supp}{\operatorname{supp}}

\def\dist{\mathop{\rm dist}\nolimits}

\def\Hess{\mathop{\rm Hess}\nolimits}

\def\<{\langle}
\def\>{\rangle}

\def\res{\mathop{\rm Res}\nolimits}

\makeatletter
 \@addtoreset{equation}{section}
 \makeatother
 
\title{Counter-examples to the fractal Weyl law for semiclassical resonances}

\author[J.-F. Bony]{Jean-Fran\c{c}ois Bony}
\address{Jean-Fran\c{c}ois Bony, IMB, CNRS (UMR 5251), Universit\'e de Bordeaux, 33405 Talence, France}
\email{bony@math.u-bordeaux.fr}
\author[S. Fujii\'e]{Setsuro Fujii\'e}
\address{Setsuro Fujii\'e, Department of Mathematical Sciences, Ritsumeikan University, 1-1-1 Noji-Higashi, Kusatsu, 525-8577 Japan}
\email{fujiie@fc.ritsumei.ac.jp}
\author[T. Ramond]{Thierry Ramond}
\address{Thierry Ramond, Universit\'e Paris-Saclay, CNRS, Laboratoire de math\'ematiques d'Orsay, 91405, Orsay, France}
\email{thierry.ramond@universite-paris-saclay.fr}
\author[M. Zerzeri]{Maher Zerzeri}
\address{Maher Zerzeri, Universit\'e Sorbonne Paris-Nord, LAGA, CNRS (UMR 7539), 93430 Villetaneuse, France}
\email{zerzeri@math.univ-paris13.fr}


\thanks{\textbf{Acknowledgments:} The first author was partially supported by ANR grant ANR-24-CE40-2939-01. The second author was partially supported by the JSPS KAKENHI Grant 24K06790 and thanks the Universities of Bordeaux and Paris 13 for their hospitality.}

\begin{document}

\begin{abstract}
Under general assumptions, the numbers of semiclassical resonances is known to be bounded from above by a negative power of $h$ which is given by the fractal dimension of the trapped set. In this paper we provide examples of operators with much less resonances, showing that these upper bounds are not always sharp.
\end{abstract}

\maketitle

\section{Introduction} \label{s1}

We study the distribution of resonances for semiclassical operators. More precisely, we are interested in the asymptotic behavior of their counting function. In hyperbolic settings, it is nowadays well established that upper bounds for the counting function of resonances are given by Weyl type expressions whose order is given by the fractal dimension of the trapped set (see e.g. \cite{Sj90_01}). Analogous results have also been obtained for resonances defined in geometric or arithmetic settings \cite{Na17_01}. On the other hand, no general lower bounds of the same order has been proven yet, that is to say no general fractal Weyl law for resonances is known. Nevertheless, in all the cases where computations of resonances are feasible either theoretically or numerically, the Weyl type upper bounds are sharp (see e.g. \cite{Bo14_01,StZw04_01}). Our aim here is to provide examples where the number of resonances is much less than predicted by the fractal Weyl law.

We consider a semiclassical Schr\"{o}dinger operator with a metric and a magnetic potential. More precisely, we study the semiclassical Weyl quantization $P = \Op ( p )$ on $\R^{n}$, $n \geq 1$, i.e.
\begin{equation} \label{a49}
P u ( x ) = ( 2\pi h )^{- n} \iint e^{i ( x - y )\cdot \xi / h } p \Big( \frac{x + y}{2} , \xi \Big) u ( y ) \, d y\, d \xi , 
\end{equation}
of a symbol $p ( x , \xi ) \in C^{\infty} ( \R^{2 n} ; \R )$ of the form
\begin{equation} \label{a15}
p ( x , \xi ) = a ( x ) \xi \cdot \xi + b ( x ) \cdot \xi + c( x ) ,
\end{equation}
where the symmetric $n \times n$ matrix $a ( x )$, the $n$ vector $b ( x ) \in \R^{n}$ and the scalar $c ( x ) \in \R$ are real and smooth. Moreover, we assume that $a$ is uniformly elliptic (i.e. there exists $\delta > 0$ such that $a( x ) \geq \delta$ for all $x \in \R^{n}$) and that $p ( x , \xi )$ is a compactly supported perturbation of $\xi^{2}$, that is $a ( x ) - \Id$, $b ( x )$ and $c ( x )$ have compact support. Thus, $P$ is a semiclassical elliptic differential operator of order $2$ which coincides with $- h^{2} \Delta$ outside a compact set. Its resonances near the real axis are well defined through the analytic distortion method or as the poles of the meromorphic extension of a cutoff resolvent (see the book of Dyatlov and Zworski \cite{DyZw19_01}). In the sequel, $\res ( P )$ denotes the set of resonances of $P$ repeated according to their algebraic multiplicity. We will need to consider more general symbols than in \eqref{a15} in Section \ref{s9}.

The Hamiltonian vector field associated to the symbol $p ( x , \xi )$ is
\begin{equation*}
H_{p} = \partial_{\xi} p \cdot \partial_{x} - \partial_{x} p \cdot \partial_{\xi} ,
\end{equation*}
on the cotangent space $T^{*} \R^{n} = \R^{2 n}$. The integral curves $t \longmapsto \exp ( t H_{p} )( x , \xi )$ of $H_{p}$ are called Hamiltonian trajectories, and $p$ is constant along such curves. The trapped set at energy $E_{0}$ is defined by
\begin{equation} \label{a92}
K ( E_{0} ) = \big\{ ( x , \xi ) \in p^{- 1} ( E_{0} ) ; \ t \longmapsto \exp ( t H_{p} ) ( x , \xi ) \text{ is bounded} \big\} .
\end{equation}
For $E_{0} > 0$, $K ( E_{0} )$ is a compact set stable by the Hamiltonian flow. The Bohr's correspondence principle predicts a strong link between the structure of the trapped set and the distribution of resonances in the semiclassical limit $h \to 0$. In this paper, we will compare the fractal dimension of the trapped set and the counting function of resonances.

Under the global analyticity of $P$ and the assumption of the existence of an escape function which vanishes at order $2$ at the trapped set, Sj\"{o}strand \cite{Sj90_01} has proved the following Weyl fractal upper bound
\begin{equation} \label{a46}
\# \big(\res ( P ) \cap [ E_{0} - \delta , E_{0} + \delta ] + i [ - \varepsilon, 0 ]\big) \leq C \varepsilon^{- d / 2} ( \varepsilon / h )^{n} ,
\end{equation}
for $C_{0} h \leq \varepsilon \leq 1 / C_{0}$ and $h$ small enough. Here $d\in [ 0 , n ]$ denotes the upper packing dimension of $K ( [ E_{0} - 2 \delta , E_{0} + 2 \delta ] )$, modulo an arbitrary small positive constant if this set is not of ``pure dimension''. We recall that the Hausdorff, Minkowski and packing dimensions are different but closely related fractal dimensions (see Mattali \cite{Ma95_01} and \eqref{a47}). Applying \eqref{a46} to the case $\varepsilon = A h$, which is adapted to the non-globally analytic setting, we get
\begin{equation} \label{a48}
\# \big(\res ( P ) \cap [ E_{0} - \delta , E_{0} + \delta ] + i [ - A h , 0 ]\big) \leq C h^{- d / 2} .
\end{equation}
The corresponding lower bound (and then the Weyl law) is neither proven nor conjectured in this broad setting. Nevertheless, to our knowledge, the upper bound \eqref{a48} is sharp and gives the correct behavior of the counting function of resonances in all the geometric settings where the semiclassical asymptotic of resonances is known (well in the island, critical points, hyperbolic closed trajectories, trapped set with symmetry, \ldots).

After this seminal work, other fractal upper bounds have been obtained in different settings and in smaller windows in real part. For semiclassical  operators as  in \eqref{a49}--\eqref{a15} (and even for pseudodifferential operators),  Sj\"{o}strand and Zworski \cite{SjZw07_02} have proven that, under a hyperbolicity assumption at the trapped set $K ( E_{0} )$,
\begin{equation} \label{a51}
\#\big(\res ( P ) \cap B ( E_{0} , A h )\big) \leq C h^{- \frac{m - 1}{2}} .
\end{equation}
for all $A > 0$. Here, $m$ denotes the upper Minkowski dimension of $K ( E_{0} )$, modulo an arbitrary small positive constant if this set is not of ``pure dimension''. Note that we rediscover formally \eqref{a48} by summing \eqref{a51} over $h^{- 1}$ balls of size $h$ assuming $d = m + 1$. In the case of several convex obstacles, Nonnenmacher, Sj\"{o}strand and Zworski \cite{NoSjZw14_01} have obtained a result similar to \eqref{a51} for the high frequency behavior of the counting function of scattering poles. For convex co-compact hyperbolic quotients, the counting function of resonances and of zeros of the Selberg zeta function follow fractal upper bounds similar to \eqref{a48} where the trapped set is replaced by the limit set of the Schottky group (see Guillop\'{e}, Lin and Zworski \cite{GuLiZw04_01}). For more general asymptotically hyperbolic manifolds with hyperbolic trapped sets, Datchev and Dyatlov \cite{DaDy13_01} have obtained a result similar to the refined upper bound \eqref{a51}. In the context of semiclassical transfer operators, Arnoldi, Faure and Weich \cite{ArFaWe17_01} have shown a fractal upper bound of the number of Ruelle resonances. We send the readers to the nice presentation of Naud \cite{Na17_01} for an instructive introduction and for more references on the fractal Weyl law for resonances.

Even if there seems to be no result on the optimality of the previous fractal upper bounds, the Weyl law for the resonances is known to be false for quantum graphs. More precisely, let $\Gamma = ( \CV , \CE )$ be a graph with a finite number of vertices $\CV$ and edges $\CE$ where $\CE$ is the union of $N$ internal edges (between two vertices) of length $\ell_{j}$, $j = 1 , \ldots , N$, and $M$ external semi-infinite edges (starting from one vertex). On $L^{2} ( \CE )$, let $P = - \partial_{x}^{2}$ on each edge with Kirchhoff type boundary condition at each vertex. In this setting, Davies and Pushnitski \cite{DaPu11_01}, and Davies, Exner and Lipovsk\'{y} \cite{DaExLi10_01} have proved that
\begin{equation} \label{a52}
\# \big(\res ( P ) \cap B ( 0 , R )\big)  = \frac{2}{\pi} W R + \CO ( 1 ) ,
\end{equation}
as $R \to + \infty$ for some $0 \leq W \leq W_{0} = \ell_{1} + \cdots + \ell_{N}$. Moreover, they provide examples where $W < W_{0}$ and even $W = 0$ (see Example 5.2 of \cite{DaExLi10_01}). If $0 < W < W_{0}$, $P$ does not follow the Weyl law in the sense that the leading coefficient is smaller than the one expected classically. But, $P$ satisfies both the fractal upper bound and the fractal lower bound. If $W = 0$, $P$ does not satisfy the expected fractal lower bound (which is of order $R$).

Now we state our main result, it shows that the fractal upper bounds of kind \eqref{a48} are not sharp in general. More precisely,

\begin{theorem}\sl \label{a1}
For any $d \in [1 , n]$, there exist operators $P = \Op ( p )$ in settings generalizing \eqref{a49}--\eqref{a15} and positive numbers $E_ {0} , \delta$ such that

$i)$ the Hausdorff, Minkowski and packing dimensions of $K ( [ E_{0} - \delta , E_{0} + \delta ] )$ are precisely $d$.

$ii)$ for any $A > 0$, there exists $C > 0$ such that 
\begin{equation*}
\# \big(\res ( P ) \cap [ E_{0} - \delta , E_{0} + \delta ] + i [ - A h , 0 ]\big) \leq C ,
\end{equation*}
for $h$ small enough.

$iii)$ there exists a smooth function $G ( x , \xi )$ on $\R^{2 n}$ with $G ( x , \xi ) = x \cdot \xi$ outside a compact set such that
\begin{equation*}
H_{p} G ( \rho ) \gtrsim \min \big( \dist \big( \rho , K ( [ E_{0} - \delta , E_{0} + \delta ] ) \big)^{2} , 1 \big) ,
\end{equation*}
for $\rho \in p^{- 1} ( [ E_{0} - \delta , E_{0} + \delta ] )$.
\end{theorem}

If the upper bound \eqref{a48} was also a lower bound for the counting function of resonances, we would have like $h^{- d / 2}$ resonances in $[ E_{0} - \delta , E_{0} + \delta ] + i [ - A h , 0 ]$. Therefore, this type of upper bounds are generally not sharp and the number of resonances is not always given by the dimension of the trapped set.

The idea behind Theorem \ref{a1} is that all the trapped trajectories do not play the same role in general. Some of them ``create'' the resonances whereas others only ``transport'' the resonant states. Then, we construct operators with few of the first ones and a lot of the second ones. On the contrary, in many classical examples (well in the island situation, hyperbolic setting with symmetry, \ldots), all the trapped trajectories have a similar role and create resonances. Note that the semiclassical asymptotic of resonances appearing in Theorem \ref{a1} is given by Theorem 6.1 of \cite{BoFuRaZe18_01}.

For the operators given by Theorem \ref{a1}, the energies in a pointed neighborhood of $E_{0}$ are non-trapping. In other words,
\begin{equation*}
K ( [ E_{0} - \delta , E_{0} + \delta ] ) = K ( E_{0} ) .
\end{equation*}
In our examples, the trapped set $K ( E_{0} )$ is the union of two fixed points and a set of (non-constant) Hamiltonian curves. Since there is at least one non-constant curve, we have $d \geq 1$. On the other hand, these curves belong to a Lagrangian manifold and then $d \leq n$. When $d \in \N$, the trapped set $K ( E_{0} )$ can contain an open subset of a $d$-dimensional manifold.

The structure of the trapped set behind Theorem \ref{a1} is rather simple and explained in Section \ref{s6}. We then construct an escape function which is not constant in the ``transport'' part of this trapped set in Section \ref{s7}. This allows us to finish the proof of the theorem in Section \ref{s8}. Unfortunately, the realization of this trapped set with a Schr\"{o}dinger operator as in \eqref{a49}--\eqref{a15} is impossible (see Section \ref{s4}). But we construct concrete examples of operators satisfying Theorem \ref{a1} as Schr\"{o}dinger operators with an absorbing potential in Section \ref{s3}, selfadjoint differential operators of order $4$ in Section \ref{s2} and matrix-valued selfadjoint Schr\"{o}dinger operators in Section \ref{s5}. In all these cases, the resonances can be defined through the complex distortion method.

\section{Resonance upper bound}

\subsection{The geometric setting} \label{s6}

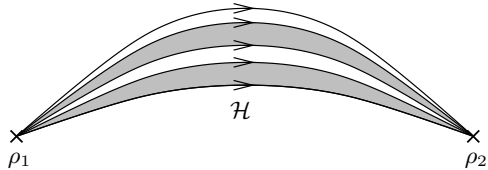
\begin{figure}
\begin{center}
\begin{picture}(0,0)%
\includegraphics{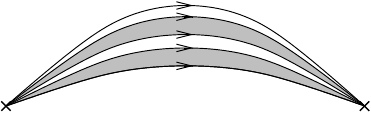}%
\end{picture}%
\setlength{\unitlength}{987sp}%
\begingroup\makeatletter\ifx\SetFigFont\undefined%
\gdef\SetFigFont#1#2#3#4#5{%
  \reset@font\fontsize{#1}{#2pt}%
  \fontfamily{#3}\fontseries{#4}\fontshape{#5}%
  \selectfont}%
\fi\endgroup%
\begin{picture}(11863,4090)(3332,-3575)
\put(3601,-3511){\makebox(0,0)[b]{\smash{{\SetFigFont{9}{10.8}{\rmdefault}{\mddefault}{\updefault}$\rho_{1}$}}}}
\put(15076,-3511){\makebox(0,0)[b]{\smash{{\SetFigFont{9}{10.8}{\rmdefault}{\mddefault}{\updefault}$\rho_{2}$}}}}
\put(9151,-2386){\makebox(0,0)[b]{\smash{{\SetFigFont{9}{10.8}{\rmdefault}{\mddefault}{\updefault}$\CH$}}}}
\end{picture}%
\end{center}
\caption{The trapped set $K ( E_{0} )$.} \label{f6}
\end{figure}

We prove the upper bound on the number of resonances given in Theorem \ref{a1} $ii)$ in the setting \eqref{a49}--\eqref{a15}, the generalization to the examples of Section \ref{s9} being straightforward. In the sequel, we consider an energy $E_{0} > 0$ and we assume that the trapped set at energy $E_{0}$ is given by
\begin{equation} \label{a17}
K ( E_{0} ) = \{ \rho_{1} , \rho_{2} \} \cup \CH .
\end{equation}
Here, $\rho_{j} = ( x_{j} , \xi_{j} ) \in \R^{2 n}$, $j = 1 , 2$, are hyperbolic fixed points of $H_{p}$. It means that
\begin{equation} \label{a53}
p ( x , \xi ) = E_{0} + \sum_{k = 1}^{n} \frac{\lambda_{j , k}}{2} \big( ( \xi_{k} - \xi_{j , k} )^{2} - ( x_{k} - x_{j , k} )^{2}\big ) + \CO \big( ( x - x_{j} , \xi - \xi_{j} )^{3} \big) ,
\end{equation}
near $\rho_{j}$ with $\lambda_{j , k} > 0$ after a possible symplectomorphism. The set $\CH$ is the union of the heteroclinic trajectories starting from $\rho_{1}$ and ending to $\rho_{2}$, that is Hamiltonian trajectories $t \longmapsto \rho ( t ) = \exp ( t H_{p} ) ( \rho )$ such that $\rho ( t ) \to \rho_{1}$ as $t \to - \infty$ and $\rho ( t ) \to \rho_{2}$ as $t \to + \infty$ (see Figure \ref{f6}). Assumption \eqref{a17} is equivalent to $K ( E_{0} ) = \Lambda_{+}^{1} \cap \Lambda_{-}^{2}$ where $\Lambda_{+}^{1}$ is the outgoing stable manifold associated to $\rho_{1}$ and $\Lambda_{-}^{2}$ is the incoming stable manifold associated to $\rho_{2}$. In particular, there is no heteroclinic trajectory from $\rho_{2}$ to $\rho_{1}$.

\subsection{Construction of an escape function} \label{s7}

In this part, we build an escape function which is increasing along the Hamiltonian curves outside $K ( E_{0} )$ and along $\CH$. Such a construction is possible only because there is no heteroclinic trajectory going back from $\rho_{2}$ to $\rho_{1}$.

\begin{proposition}\sl \label{a9}
There exist $\delta > 0$ and a smooth function $G ( x , \xi )$ on $\R^{2 n}$ with $G ( x , \xi ) = x \cdot \xi$ outside a compact set such that
\begin{equation} \label{a8}
H_{p} G ( \rho ) \gtrsim \min \big( \dist ( \rho , \{ \rho_{1} , \rho_{2} \} )^{2} , 1 \big) ,
\end{equation}
for $\rho \in p^{- 1} ( [ E_{0} - \delta , E_{0} + \delta ] )$.
\end{proposition}

In particular, Proposition \ref{a9} implies Theorem \ref{a1} $iii)$.

\begin{proof}
Near each hyperbolic fixed point $\rho_{j} = ( x_{j} , \xi_{j} )$, $j=1 , 2$, we choose the function 
\begin{equation*}
G_{j} ( x , \xi ) = ( x - x_{j} ) \cdot ( \xi - \xi_{j} ) + C_{j} ,
\end{equation*}
where $C_{1} = 0$ and $C_{2} \gg 1$ are constants. From \eqref{a53}, there exists $\varepsilon > 0$ such that, for all $\rho = ( x , \xi )$ in $B ( \rho_{j} , \varepsilon ) = \{ \rho \in \R^{2 n} ; \ \vert \rho - \rho_{j} \vert < \varepsilon \}$, we have
\begin{equation} \label{a3}
H_{p} G_{j} ( \rho ) \gtrsim \vert \rho - \rho_{j} \vert^{2} .
\end{equation}
We now follow the shooting method by connecting $G_{1}$ and $G_{2}$ using an increasing function along $\CH$ for $C_{2}$ large enough.

Note that, when $0$ is a hyperbolic fixed point of $H_{p}$, the function $t \longmapsto \vert \rho ( t ) \vert^{2}$ is non-decreasing (resp. non-increasing) if $\rho \in \Lambda_{+}$ (resp. $\rho \in \Lambda_{-}$) as long as $\rho ( t)$ stays in a neighborhood of $0$. To prove this point, it is enough to compute the time derivative of $\vert \rho ( t ) \vert^{2}$. After shrinking possibly $\varepsilon$ and since $\CH \cap S ( \rho_{1} , \varepsilon / 4 )$ is a compact set, there exists $T > 0$ such that, for all $\rho \in \CH \cap S ( \rho_{1} , \varepsilon / 4 )$,
\begin{equation*}
\left\{ \begin{aligned}
&\text{for all } t \geq T , \text{ we have } \rho ( t ) \in B ( \rho_{2} , \varepsilon / 4 ) ,  \\
&\text{if } \rho ( t ) \in B ( \rho_{1} , 7 \varepsilon / 16 ) \text{ for  }t \geq 0 \text{ then } \rho ( s ) \in B ( \rho_{1} , 7 \varepsilon / 16 ) \text{ for all } s \leq t ,   \\
&\text{if } \rho ( t ) \in B ( \rho_{2} , 7 \varepsilon / 16 ) \text{ for } t \geq 0 \text{ then } \rho ( s ) \in B ( \rho_{2} , 7 \varepsilon / 16 ) \text{ for all } t \leq s .
\end{aligned} \right.
\end{equation*}
Since $\CH \cap S ( \rho_{1} , \varepsilon / 4 )$ is compact and $1 / 4 < 3 / 8 < 7 / 16 < 1 / 2$, there exists $\Omega$, an open neighborhood of $\CH \cap S ( \rho_{1} , \varepsilon / 4 )$ contained in $B ( \rho_{1} , 3 \varepsilon / 8 )$, such that, for all $\rho \in \Omega$,
\begin{equation*}
\left\{ \begin{aligned}
&\text{we have } \rho ( T ) \in B ( \rho_{2} , \varepsilon / 4 ) ,  \\
&\text{if } \rho ( t ) \in B ( \rho_{1} , 3 \varepsilon / 8 ) \text{ for  } 0 \leq t \leq T , \text{ then } \rho ( s ) \in B ( \rho_{1} , \varepsilon / 2 ) \text{ for all } 0 \leq s \leq t ,   \\
&\text{if } \rho ( t ) \in B ( \rho_{2} , 3 \varepsilon / 8 ) \text{ for } 0 \leq t \leq T , \text{ then } \rho ( s ) \in B ( \rho_{2} , \varepsilon / 2 ) \text{ for all } t \leq s \leq T .
\end{aligned} \right.
\end{equation*}

For $j = 1 , 2$, let $\chi_{j} \in C^{\infty}_{0} ( B ( \rho_{j} , \varepsilon ) ; [ 0 , 1 ] )$ with $\chi_{j} = 1$ near $B ( \rho_{j} , 3 \varepsilon / 4 )$. We also consider a function $\chi_{0} \in C^{\infty}_{0} ( \R^{2 n} \setminus ( B ( \rho_{1} , \varepsilon / 2 ) \cup B ( \rho_{2} , \varepsilon / 2 ) ) ; [ 0 , 1 ] )$ with $\chi_{0} = 1$ near
\begin{equation*}
\big\{ \rho ( t ) ; \ \rho \in \Omega , \ t \in [ 0 , T ] \big\} \setminus \big( B ( \rho_{1} , 3 \varepsilon / 4 ) \cup B ( \rho_{2} , 3 \varepsilon / 4 ) \big) .
\end{equation*}
For $\rho \in \big\{ \omega ( t ) ; \ \omega \in \Omega , \ t \in [ 0 , T ] \big\}$, we define the function
\begin{equation} \label{a7}
F ( \rho ) = \Big( G_{2} ( \omega ( T ) ) - G_{1} ( \omega ( 0 ) ) - \int_{0}^{T} ( \chi_{1} H_{p} G_{1} + \chi_{2} H_{p} G_{2} ) ( \omega ( s ) ) \, d s \Big) / \int_{0}^{T} \chi_{0} ( \omega ( s ) ) \, d s ,
\end{equation}
where $\rho = \omega ( t )$ for some $\omega \in \Omega$ and $t \in [ 0 , T ]$. From the choice of the cut-off functions $\chi_{\bullet}$, $F$ is well-defined, smooth and constant along the Hamiltonian curves from $\Omega$.

We then solve the problem
\begin{equation*}
\left\{ \begin{aligned}
&H_{p} G_{0} = \chi_{1} H_{p} G_{1} + \chi_{0} F + \chi_{2} H_{p} G_{2}  &&\text{ near } \CH \setminus \big( B ( \rho_{1} , \varepsilon / 4 ) \cup B ( \rho_{2} , \varepsilon / 4 ) \big) ,  \\
&G_{0} = G_{1} &&\text{ on } S ( \rho_{1} , \varepsilon / 4 ) .
\end{aligned} \right.
\end{equation*}
That is
\begin{equation} \label{a2}
G_{0} ( \omega ( t ) ) = G_{1} ( \omega ( 0 ) ) + \int_{0}^{t} \big( \chi_{1} H_{p} G_{1} + \chi_{0} F + \chi_{2} H_{p} G_{2} \big) ( \omega ( s ) ) \, d s ,
\end{equation}
for $\omega \in \Omega$ and $0 \leq t \leq T$. For $\rho = \omega ( t )$, with $\omega \in \Omega$ and $0 \leq t \leq T$, near $S ( \rho_{1} , \varepsilon / 4)$, \eqref{a2} becomes
\begin{equation*}
G_{0} ( \rho ) = G_{1} ( \omega ( 0 ) ) + \int_{0}^{t} H_{p} G_{1} ( \omega ( s ) ) \, d s  = G_{1} ( \rho ) ,
\end{equation*}
since $\chi_{0} ( \omega ( s ) ) = \chi_{2} ( \omega ( s ) ) = 0$ and $\chi_{1} ( \omega ( s ) ) = 1$ for $s \in [ 0 , t ]$. On the other hand, for $\rho = \omega ( t )$, with $\omega \in \Omega$ and $0 \leq t \leq T$, near $S ( \rho_{2} , \varepsilon / 4)$, \eqref{a2} gives
\begin{equation*}
G_{0} ( \rho ) = G_{1} ( \omega ( 0 ) ) + \int_{0}^{T} ( \chi_{1} H_{p} G_{1} ) ( \omega ( s ) ) \, d s  + F ( \rho ) \int_{0}^{T} \chi_{0} ( \omega ( s ) ) \, d s  + \int_{0}^{t} ( \chi_{2} H_{p} G_{2} \big) ( \omega ( s ) ) \, d s ,
\end{equation*}
since $\chi_{0} ( \omega ( s ) ) = \chi_{1} ( \omega ( s ) ) = 0$ for $s \in [ t , T ]$. Using \eqref{a7}, it becomes
\begin{align*}
G_{0} ( \rho ) = G_{2} ( \omega ( T ) )  + \int_{T}^{t} ( \chi_{2} H_{p} G_{2} \big) ( \omega ( s ) ) \, d s = G_{2} ( \omega ( T ) )  + \int_{T}^{t} H_{p} G_{2} ( \omega ( s ) ) \, d s  = G_{2} ( \rho ) ,
\end{align*}
since $\chi_{2} ( \omega ( s ) ) = 1$ for $s \in [ t , T ]$.

Thus, if we glue $G_{0}$ with $G_{1}$ in $B ( \rho_{1} , \varepsilon / 4 )$ and with $G_{2}$ in $B ( \rho_{2} , \varepsilon / 4 )$, we obtain a smooth function (still denoted $G_{0}$) defined in a neighborhood of $K ( E_{0} )$ and such that
\begin{equation} \label{a4}
H_{p} G_{0} = \chi_{1} H_{p} G_{1} + \chi_{0} F + \chi_{2} H_{p} G_{2} ,
\end{equation}
near $K ( E_{0} )$. For $C_{2} > 0$ fixed large enough, \eqref{a7} gives that $F ( \rho ) \geq 1$ on $\CH$. From the properties of the functions $\chi_{\bullet}$, \eqref{a3} and \eqref{a4}, we eventually obtain
\begin{equation} \label{a5}
H_{p} G_{0} ( \rho ) \gtrsim \dist ( \rho , \{ \rho_{1} , \rho_{2} \} )^{2} ,
\end{equation}
for $\rho$ near $K ( E_{0} )$.

We now glue $G_{0}$ with an escape function outside the trapped set, using a construction of G\'erard and Sj\"{o}strand \cite{GeSj87_01}. Let $\varphi_{0} \prec \psi_{0} \in C^{\infty}_{0} ( \R^{2 n} ; [ 0 , 1 ] )$ be such that $\varphi_{0} = 1$ near $K ( E_{0} )$ and $\psi_{0}$ is supported in the region where \eqref{a5} holds true. From \cite[Appendix]{GeSj87_01}, there exist a smooth function $G_{\infty} \in C^{\infty} ( \R^{2 n} )$ and $\delta > 0$ such that $G_{\infty} ( x , \xi ) = x \cdot \xi$ outside a compact set, $G_{\infty} = 0$ near $K ( E_{0} )$ and
\begin{equation} \label{a6}
H_{p} G_{\infty} \gtrsim \left\{ \begin{aligned}
&0 &&\text{ in } p^{- 1} ( [ E_{0} - \delta , E_{0} + \delta ] ) , \\
&1 &&\text{ in } p^{- 1} ( [ E_{0} - \delta , E_{0} + \delta ] ) \setminus \varphi_{0}^{- 1} ( 1 ) .
\end{aligned} \right.
\end{equation}
We then set
\begin{equation*}
G = \nu \psi_{0} G_{0} + G_{\infty} ,
\end{equation*}
with $\nu > 0$ chosen later. Using \eqref{a5}, \eqref{a6}, the properties of $\psi_{0}$ and
\begin{equation*}
H_{p} G = \nu \psi_{0} H_{p} G_{0} + \nu G_{0} H_{p} \psi_{0} + H_{p} G_{\infty} ,
\end{equation*}
we get \eqref{a8} for $\nu$ small enough. Moreover, $G ( x , \xi ) = x \cdot \xi$ outside a compact set and $G = \nu G_{0}$ near $K ( E_{0} )$.
\end{proof}

\subsection{End of the proof} \label{s8}

Using the escape function of Proposition \ref{a9}, we prove Theorem \ref{a1} $ii)$. At this point, the reasoning follows the general strategy of Sj\"{o}strand (see e.g. \cite{Sj90_01,Sj97_01}) to get upper bound for the counting function of resonances. We decompose the proof in two parts. The first step is to show that, for all $C_{0} > 0$, there exists $C > 0$ such that
\begin{equation} \label{a54}
\# \big( \res ( P ) \cap B ( E_{0} , C_{0} h ) \big) \leq C ,
\end{equation}
for $h$ small enough. When the trapped set consists of a unique hyperbolic fixed point, this follows in this $C^{\infty}$ setting from Proposition 4.2 of \cite{BoFuRaZe07_01} applied to the weighted operator $Q_{z}$ of Section 3 of \cite{BoFuRaZe11_01}. The escape function of Proposition \ref{a9} allows to ``remove'' (conjugating by polynomially bounded weights) the heteroclinic set $\CH$ from $K ( E_{0} )$ which then reduces to the two hyperbolic fixed points $\rho_{1} , \rho_{2}$.  We can then adapt the arguments of \cite{BoFuRaZe07_01,BoFuRaZe11_01} to get \eqref{a54}.

It remains to show that there exists $\delta > 0$ such that, for all $A > 0$, there exists $C > 0$ such that $P$ has no resonances in
\begin{equation*}
[ E_{0} - \delta , E_{0} + \delta ] + i [ - A h , 0 ] \setminus B ( E_{0} , C h ) ,
\end{equation*}
for $h$ small enough. In the case of a unique hyperbolic fixed point, this has been proved in Theorem 3.1 i) of \cite{BoFuRaZe11_01}. Thanks to the escape function given by Proposition \ref{a9}, this proof can be directly adapted to the present setting and gives the required resonance free region (see Section 3.4 of \cite{BoFuRaZe11_01}).

\section{Examples} \label{s9}

In this part, we construct semiclassical operators compatible with Theorem \ref{a1} and satisfying the geometric configuration \eqref{a17}. It is natural to try to achieve this goal with the simplest example of operators: the semiclassical Schr\"{o}dinger operators with an electric potential. That is
\begin{equation} \label{a50}
P = \Op ( p ) = - h^{2} \Delta + V ( x ) ,
\end{equation}
with $p  ( x , \xi ) = \xi^{2} + V ( x )$ and $V \in C^{\infty}_{0} ( \R^{n} ; \R )$. In that case, the two fixed points are necessarily of the form $\rho_{j} = ( x_{j} , 0 )$ for $j = 1 , 2$. Moreover, if $t \longmapsto ( x ( t ) , \xi ( t ) )$ is an Hamiltonian trajectory, so is $t \longmapsto ( x ( - t ) , - \xi ( - t ) )$. Thus, if $\CH \neq \emptyset$, there exists also a Hamiltonian trajectory from $\rho_{2}$ to $\rho_{1}$ which is impossible under \eqref{a17}. This shows that there is no interesting example of the form \eqref{a50} satisfying \eqref{a17}.

\Subsection{The case of Schr\"{o}dinger operators} \label{s4}

Since it is not possible to construct an example with an electric potential only, it is natural to try by adding a magnetic potential and a metric. Thus, we consider the more general setting of \eqref{a15}.

\begin{proposition}\sl \label{a16}
There exists no scalar symbol $p ( x , \xi )$ as in \eqref{a15} which satisfies the geometric assumption \eqref{a17}.
\end{proposition}

\begin{proof}
We work by contradiction and consider a symbol $p ( x , \xi )$ as in \eqref{a15} whose trapped set at energy $E_{0} > 0$ consists of two hyperbolic fixed points $\rho_{j} = ( x_{j} , \xi_{j} )$, $j = 1 , 2$, and possible heteroclinic trajectories from $\rho_{1}$ to $\rho_{2}$. The first step is to show that $x_{1} \neq x_{2}$ by contradiction. Assume that $x_{1} = x_{2}$ and note that
\begin{equation*}
H_{p} =
\begin{pmatrix}
2 a ( x ) \xi + b ( x ) \\
- \partial_{x} a ( x ) \xi \cdot \xi - \partial_{x} b ( x ) \cdot \xi - \partial_{x} c ( x )
\end{pmatrix} .
\end{equation*}
Since $H_{p} ( x_{j} , \xi_ {j} ) = 0$, we get $2 a ( x_{1} ) \xi_{1} + b ( x_{1} ) = 0 = 2 a ( x_{1} ) \xi_{2} + b ( x_{1} )$. Using that $a ( x_{1} )$ is invertible, we deduce $\xi_{1} = \xi_{2}$ and then $\rho_{1} = \rho_{2}$, a contradiction. This proves $x_{1} \neq x_{2}$.

We will now make a change of gauge. Let $\varphi \in C^{\infty}_{0} ( \R^{n} )$ be such that $\partial_{x} \varphi ( x_{j} ) = \xi_{j}$. We define $\widetilde{p} ( x , \xi ) = p \circ \kappa ( x , \xi )$ where $\kappa ( x , \xi ) = ( x , \xi + \partial_{x} \varphi ( x ) )$ is a symplectic change of variables. In particular,
\begin{equation} \label{a11}
\widetilde{p} ( x , \xi ) = \widetilde{a} ( x ) \xi \cdot \xi + \widetilde{b} ( x ) \cdot \xi + \widetilde{c} ( x ) ,
\end{equation}
where $\widetilde{a} = a$ and $\widetilde{b} , \widetilde{c}$ satisfy the same assumptions that $b , c$. Since $\kappa$ is symplectic, the Hamiltonian curves of $p$ are transformed by $\kappa^{- 1}$ into the Hamiltonian curves of $\widetilde{p}$. Then, the flows of $H_{\widetilde{p}}$ and $H_{p}$ have the same geometry and $\rho_{j}$ becomes $\widetilde{\rho}_{j} = ( x_{j} , 0 )$. As before,
\begin{equation*}
H_{\widetilde{p}} =
\begin{pmatrix}
2 \widetilde{a} ( x ) \xi + \widetilde{b} ( x ) \\
- \partial_{x} \widetilde{a} ( x ) \xi \cdot \xi - \partial_{x} \widetilde{b} ( x ) \cdot \xi - \partial_{x} \widetilde{c} ( x )
\end{pmatrix} .
\end{equation*}
Since $H_{\widetilde{p}} ( x_{j} , 0 ) = 0$, we have $\widetilde{b} ( x_{j} ) = \partial_{x} \widetilde{c} ( x_{j} ) = 0$. Since $\widetilde{p} ( x_{j} , 0 ) = E_{0}$, we get $c ( x_{j} ) = E_{0}$. Summing up,
\begin{equation*}
\widetilde{p} ( x_{j} , \xi ) = a ( x_{j} ) \xi \cdot \xi + E_{0} ,
\end{equation*}
for all $\xi \in \R^{n}$. In particular, $\rho_{j}$ is the unique point of $\widetilde{p}^{- 1} ( E_{0} )$ above $x_{j}$, that is
\begin{equation} \label{a14}
\{ ( x , \xi ) ; \ \widetilde{p} ( x , \xi ) = E_{0} \text{ and } x = x_{j} \} = \{ ( x_{j} , 0 ) \} .
\end{equation}

Let $\widetilde{\Lambda}_{+}^{j} \subset \R^{2 n}$ denote the stable outgoing manifold associated to the hyperbolic fixed point $\widetilde{\rho}_{j}$ of $H_{\widetilde{p}}$. We now prove that
\begin{equation} \label{a10}
\widetilde{\Lambda}_{+}^{j} \text{ projects nicely on the } x\text{-space locally near } \widetilde{\rho}_{j}.
\end{equation}
For that, it is enough to prove by contradiction that the tangent space $T_{\widetilde{\rho}_{j}} \widetilde{\Lambda}_{+}^{j}$ projects nicely on the $x$-space. Then, assume that there exists $T = ( T_{x} , T_{\xi} ) \in T_{\widetilde{\rho}_{j}} \widetilde{\Lambda}_{+}^{j}$ with $T \neq 0$ and $T_{x} = 0$. Let $\gamma : [ - 1 , 1 ] \longrightarrow \widetilde{\Lambda}_{+}^{j}$ be a smooth curve such that $\gamma ( 0 ) = \widetilde{\rho}_{j}$ and $\dot{\gamma} ( 0 ) = T$. Since $\widetilde{\Lambda}_{+}^{j} \subset \widetilde{p}^{- 1} ( E_{0} )$, we have $\widetilde{p} ( \gamma ( t ) ) = E_{0}$ for all $t \in [ - 1 , 1 ]$. Differentiating this relation leads to $d_{\gamma ( t )} \widetilde{p} ( \dot{\gamma} ( t ) ) =  0$ and eventually
\begin{equation*}
\Hess_{\gamma ( t )} ( \widetilde{p} ) \dot{\gamma} ( t ) \cdot \dot{\gamma} ( t ) + d_{\gamma ( t )} \widetilde{p} ( \ddot{\gamma} ( t ) ) = 0 .
\end{equation*}
Computing at $t = 0$ and using $d_{\widetilde{\rho}_{j}} \widetilde{p} = 0$ and \eqref{a11}, we deduce $0 = \Hess_{\widetilde{\rho}_{j}} ( \widetilde{p} ) ( 0 , T_{\xi} ) \cdot ( 0 , T_{\xi} ) = 2 \widetilde{a} ( x_{j} ) T_{\xi} \cdot T_{\xi}$. Since $\widetilde{a} ( x_{j} )$ is positive, it yields $T_{\xi} = 0$ which provides a contradiction and implies \eqref{a10}.

We now use an argument of algebraic topology. For $\varepsilon > 0$ small enough, $t \geq 0$ and $\omega \in \S^{n - 1}$, we define
\begin{equation}
x ( t , \omega ) = \pi_{x} \big( \exp ( t H_{\widetilde{p}} ) \big( x_{2} + \varepsilon \omega , \xi ( \varepsilon \omega ) \big) \big) ,
\end{equation}
where $\pi_{x} ( x , \xi ) = x$ is the $x$-space projection and $\xi ( \varepsilon \omega )$ is the unique point of $\R^{n}$ such that $( x_{2} + \varepsilon \omega , \xi ( \varepsilon \omega ) ) \in \widetilde{\Lambda}_{+}^{2}$ (see \eqref{a10}). In particular, the map $( t , \omega ) \longmapsto x ( t , \omega )$ is smooth. Since $\exp ( t H_{\widetilde{p}} ) ( x_{2} + \varepsilon \omega , \xi ( \varepsilon \omega ) )$ belongs to $\widetilde{p}^{- 1} ( E_{0} )$ and does to reach $\widetilde{\rho}_{1}$ and $\widetilde{\rho}_{2}$, we deduce
\begin{equation} \label{a13}
\forall t \geq 0 , \quad \forall \omega \in \S^{n - 1}, \qquad x ( t , \omega ) \notin \{ x_{1} , x_{2} \} ,
\end{equation}
from \eqref{a14}. On the other hand, since the Hamiltonian trajectories in $\widetilde{\Lambda}_{+}^{2}$ are trapped as $t \to - \infty$ but non-trapped from \eqref{a17}, they escape to infinity at $t \to + \infty$. Using that a Hamiltonian trajectory which enters in an outgoing region never comes back (this follows for instance from Proposition \ref{a9}) and that $\S^{n - 1}$ is compact, these trajectories escape uniformly to infinity. In other words,
\begin{equation} \label{a12}
\forall R > 0 , \quad \exists T > 0 , \quad \forall t \geq T, \quad  \forall \omega \in \S^{n - 1} , \qquad \vert x ( t , \omega ) \vert \geq R .
\end{equation}
For all $t \geq 0$ and $\omega \in \S^{n  - 1}$, we set
\begin{equation}
F ( t , \omega ) = \frac{\frac{x ( t , \omega ) - x_{2}}{\vert x ( t , \omega ) - x_{2} \vert^{2}} - \frac{x_{1} - x_{2}}{\vert x_{1} - x_{2} \vert^{2}}}{\Big\vert \frac{x ( t , \omega ) - x_{2}}{\vert x ( t , \omega ) - x_{2} \vert^{2}} - \frac{x_{1} - x_{2}}{\vert x_{1} - x_{2} \vert^{2}} \Big \vert} .
\end{equation}
From \eqref{a13}, $x ( t , \omega ) - x_{2}$ does not vanish and $\frac{x ( t , \omega ) - x_{2}}{\vert x ( t , \omega ) - x_{2} \vert^{2}} - \frac{x_{1} - x_{2}}{\vert x_{1} - x_{2} \vert^{2}} = 0$, which is equivalent to $x ( t , \omega ) = x_{1}$, is impossible. Thus, $F : [ 0 , + \infty [ \times \S^{n - 1} \longrightarrow \S^{n - 1}$ is well-defined and continuous. By definition,
\begin{equation*}
F ( 0 , \omega ) = \frac{\frac{\varepsilon \omega}{\varepsilon^{2}} - \frac{x_{1} - x_{2}}{\vert x_{1} - x_{2} \vert^{2}}}{\big\vert \frac{\varepsilon \omega}{\varepsilon^{2}} - \frac{x_{1} - x_{2}}{\vert x_{1} - x_{2} \vert^{2}} \big \vert} = \omega + \CO ( \varepsilon ) ,
\end{equation*}
uniformly for $\omega \in \S^{n - 1}$. In particular, $F ( 0 , \cdot )$ is homotopic to $\Id_{\S^{n - 1}}$ for $\varepsilon$ small enough. On the other hand, \eqref{a12} implies that
\begin{equation*}
F ( t , \omega ) = \frac{x_{2} - x_{1}}{\vert x_{2} - x_{1} \vert} + o ( 1 )  ,
\end{equation*}
as $t$ goes to $+ \infty$ uniformly for $\omega \in \S^{n - 1}$. Then, $F ( T , \cdot )$ is homotopic to the constant map $f_{\infty} ( \omega ) = \frac{x_{2} - x_{1}}{\vert x_{2} - x_{1} \vert}$ for $T$ large enough. Using that $F$ provides a homotopy between $F ( 0 , \cdot )$ and $F ( T , \cdot )$ and taking the degree, which is invariant under homotopy, we eventually obtain
\begin{equation}
1 = \deg ( \Id_{\S^{n - 1}} ) = \deg ( F ( 0 , \cdot ) ) = \deg ( F ( T , \cdot ) ) = \deg ( f_{\infty} ) = 0 ,
\end{equation}
which is impossible (see \cite[Section 23]{Fu95_01} for more details on the degree). This concludes the proof of Proposition \ref{a16}.
\end{proof}

\Subsection{A Schr\"{o}dinger operator with an absorbing potential} \label{s3}

We have seen in the previous section that it is impossible to realize the geometric setting \eqref{a17} with a scalar selfadjoint elliptic differential operator of order $2$ which is a perturbation of $- h^{2} \Delta$ at infinity. Using the theory of resonances for ``black box'' operators due to Sj\"{o}strand and Zworski \cite{SjZw91_01}, one can also consider selfadjoint semiclassical pseudodifferential operators which coincide with $- h^{2} \Delta$ outside a compact set. Unfortunately, we have not been able to find such an operator satisfying \eqref{a17}. Instead, we give in this part a non-selfadjoint Schr\"{o}dinger operator in $\R^{2}$ for which \eqref{a17} holds. Such an example can be easily generalized in higher dimensions $n \geq 2$.

\begin{figure}
\begin{center}
\begin{picture}(0,0)%
\includegraphics{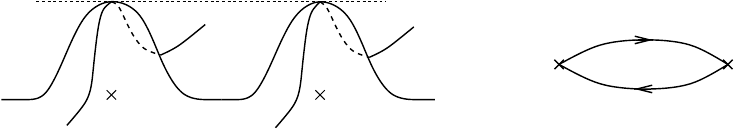}%
\end{picture}%
\setlength{\unitlength}{987sp}%
\begingroup\makeatletter\ifx\SetFigFont\undefined%
\gdef\SetFigFont#1#2#3#4#5{%
  \reset@font\fontsize{#1}{#2pt}%
  \fontfamily{#3}\fontseries{#4}\fontshape{#5}%
  \selectfont}%
\fi\endgroup%
\begin{picture}(23488,4121)(-14368,-4905)
\put(-13499,-1036){\makebox(0,0)[rb]{\smash{{\SetFigFont{9}{10.8}{\rmdefault}{\mddefault}{\updefault}$E_{0}$}}}}
\put(-10574,-4561){\makebox(0,0)[b]{\smash{{\SetFigFont{9}{10.8}{\rmdefault}{\mddefault}{\updefault}$x^{1}$}}}}
\put(-3899,-4561){\makebox(0,0)[b]{\smash{{\SetFigFont{9}{10.8}{\rmdefault}{\mddefault}{\updefault}$x^{2}$}}}}
\put(6226,-1636){\makebox(0,0)[b]{\smash{{\SetFigFont{9}{10.8}{\rmdefault}{\mddefault}{\updefault}$\gamma_{1}$}}}}
\put(6226,-4261){\makebox(0,0)[b]{\smash{{\SetFigFont{9}{10.8}{\rmdefault}{\mddefault}{\updefault}$\gamma_{2}$}}}}
\put(3601,-3511){\makebox(0,0)[b]{\smash{{\SetFigFont{9}{10.8}{\rmdefault}{\mddefault}{\updefault}$\rho_{1}$}}}}
\put(9001,-3511){\makebox(0,0)[b]{\smash{{\SetFigFont{9}{10.8}{\rmdefault}{\mddefault}{\updefault}$\rho_{2}$}}}}
\end{picture}%
\end{center}
\caption{The potential of $Q$ and the set $K_{Q} ( E_{0} )$.} \label{f4}
\end{figure}
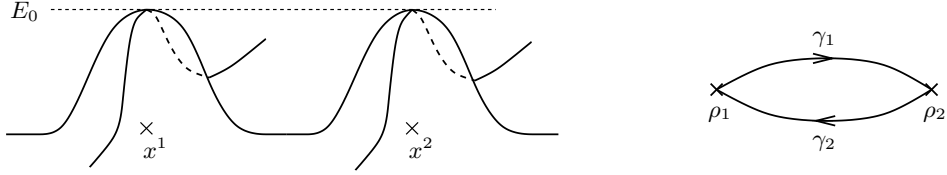

\subsubsection{The initial operator}

Let $V \in C^{\infty}_{0} ( \R^{2} )$ be a radial barrier of height $E_{0} > 0$. More precisely, we assume that $V$ is radial, $x \cdot \nabla V ( x ) < 0$ for $x$ in the interior of $\supp V \setminus \{ 0 \}$, $V$ has a non-degenerate maximum at $x = 0$ (say $V ( x ) = E_{0} - \lambda x^{2} + \CO ( x^{3} )$ near $0$ with $\lambda > 0$) and $V$ has a large scattering angle in the sense of Appendix B.3 of \cite{BoFuRaZe18_01} (roughly speaking, it means that the angle between the two asymptotic directions of every non-trapped Hamiltonian trajectory of $\xi^{2} + V ( x )$ at energy $E_{0}$ is at least $\pi / 3$). For $L \geq 1$ large enough, we define
\begin{equation*}
Q = - h^{2} \Delta + V ( x_{1} + L , x_{2} ) + V ( x_{1} - L , x_{2} ) ,
\end{equation*}
the Schr\"{o}dinger operator with double bumps. The potential of $Q$ is illustrated in Figure \ref{f4}. The trapped set at energy $E_{0}$ is given by
\begin{equation*}
K_{Q} ( E_{0} ) = \{ \rho_{1} , \rho_{2} \} \cup \gamma_{1} \cup \gamma_{2} ,
\end{equation*}
where $\rho_{1} = ( x^{1} , 0 )$ with $x^{1} = ( - L , 0 )$ and $\rho_{2} = ( x^{2} , 0 )$ with $x^{2} = ( L , 0 )$ are hyperbolic fixed points and $\gamma_{1}$ (resp. $\gamma_{2}$) is a heteroclinic trajectory from $\rho_{1}$ to $\rho_{2}$ (resp. from $\rho_{2}$ to $\rho_{1}$). The asymptotic of the resonances for this operator has been studied in Example 6.9 of \cite{BoFuRaZe18_01}.

\subsubsection{Removal of $\gamma_{2}$}

The next step is to ``remove'' $\gamma_{2}$ using an absorbing perturbation, without ``touching'' a neighborhood of $\gamma_{1}$. This idea has already been applied by Royer \cite{Ro10_01}, Datchev and Vasy \cite[Section 5.3]{DaVa12_01} or in various situations (see e.g. Examples  4.14, 5.9, 7.11, \ldots) in a previous work \cite{BoFuRaZe18_01}. Since the $x$-space projection of $\gamma_{1}$ and $\gamma_{2}$ is the same (equal to $[ x^{1} , x^{2} ]$), one can not directly use a potential or a differential operator. We propose two ways to get around this difficulty.

First, one can add an absorbing pseudodifferential operator whose symbol is supported near a point of $\gamma_{2}$. More precisely, let $w ( x , \xi ) \in C^{\infty}_{0} ( \R^{4} )$ supported in a vicinity of $( 0 , 0 , - \sqrt{E_{0}} , 0 ) \in \gamma_{2}$ with $w ( 0 , 0 , - \sqrt{E_{0}} , 0 ) \neq 0$. In particular, $w = 0$ near $\gamma_{1}$. We also take $\chi \in C^{\infty}_{0} ( \R^{2} )$ such that $w ( \cdot , \xi ) \prec \chi$ for all $\xi \in \R^{2}$. We then set
\begin{equation} \label{a41}
P = Q - i \alpha ( h ) \chi \Op ( w )^{*} \Op ( w ) \chi 
\end{equation}
for some $h \vert \ln h \vert \leq \alpha ( h ) \leq 1$. The non-negative pseudodifferential operator $\chi \Op ( w )^{*} \Op ( w ) \chi$ is elliptic near $( 0 , 0 , - \sqrt{E_{0}} , 0 ) $ and negligible near $\gamma_{1}$. Then, the trapped set of $P$ at energy $E_{0}$ becomes
\begin{equation*}
K_{P} ( E_{0} ) = \{ \rho_{1} , \rho_{2} \} \cup \gamma_{1} ,
\end{equation*}
and \eqref{a17} holds true with $\CH = \gamma_{1}$. The presence of the cut-off function $\chi$ guarantees that $P$ is a compactly supported perturbation of $- h^{2} \Delta$ and that its resonances can be defined through the theory of ``black box'' operators (see Sj\"{o}strand and Zworski \cite{SjZw91_01}).

Another approach is to change the $x$-space projection of $\gamma_{2}$ using a small magnetic field and then to add an absorbing potential. For that, let us define
\begin{equation} \label{a31}
\widetilde{Q} = \Op ( \widetilde{q} ) \qquad \text{ with } \qquad \widetilde{q} ( x , \xi ) = q ( x , \xi ) + \varepsilon ( \xi_{1} - \sqrt{E_{0}} ) x_{2} \psi ( x ) ,
\end{equation}
where
\begin{equation*}
q ( x , \xi ) = \xi^{2} + V ( x_{1} + L , x_{2} ) + V ( x_{1} - L , x_{2} ) ,
\end{equation*}
is the symbol of $Q$, $\varepsilon > 0$ and $\psi \in C^{\infty}_{0} ( \R^{2} )$ is supported near $0$ and $\psi ( 0 ) \neq 0$. Using the stability of transversal heteroclinic trajectories, we claim that, for $\varepsilon > 0$ small enough, there is exactly one heteroclinic trajectory $\gamma_{1}^{\varepsilon}$ (resp. $\gamma_{2}^{\varepsilon}$) from $\rho_{1}$ to $\rho_{2}$ (resp. from $\rho_{2}$ to $\rho_{1}$) which is close to $\gamma_{1}$ (resp. $\gamma_{2}$). Indeed, for $\varepsilon$ small enough, the outgoing (resp. incoming) stable manifold from $\rho_{1}$ (resp. $\rho_{2}$) projects nicely on the $x$-space near $0$. Let $\varphi_{+}^{1} ( x , \varepsilon )$ (resp. $\varphi_{-}^{2} ( x , \varepsilon )$) denotes its generating function vanishing at $x^{1}$ (resp. $x^{2}$) which is smooth in the $( x , \varepsilon )$ variables. Since $\widetilde{q} ( x , \nabla \varphi_{+}^{1} ( x , \varepsilon ) ) = \widetilde{q} ( x , \nabla \varphi_{-}^{2} ( x , \varepsilon ) ) = E_{0}$ and $\partial_{x_{1}} \varphi_{+}^{1} ( 0 , 0 ) =\partial_{x_{1}} \varphi_{-}^{2} ( 0 , 0 ) = \sqrt{E_{0}}$, a point $ ( 0 , x_{2} )$ belongs to a heteroclinic trajectory of $\widetilde{q}$ iff $\partial_{x_{2}} \varphi_{+}^{1} ( 0 , x_{2} , \varepsilon ) = \partial_{x_{2}} \varphi_{-}^{2} ( 0 , x_{2} , \varepsilon )$. Moreover, we have $\partial_{x_{2}} \varphi_{+}^{1} ( 0 , x_{2} , 0 ) = \partial_{x_{2}} \varphi_{-}^{2} ( 0 , x_{2} , 0 )$ and $\partial_{x_{2}}^{2} \varphi_{+}^{1} ( 0 , x_{2} , 0 ) \neq \partial_{x_{2}}^{2} \varphi_{-}^{2} ( 0 , x_{2} , 0 )$ (i.e. the intersection of the Lagrangian manifolds is transversal) since the functions $\varphi_{+}^{1} ( x , 0 )$ (resp. $\varphi_{-}^{2} ( x , 0 )$) are radial with respect to $x^{1}$ (resp. $x^{2}$). Eventually the claim for $\gamma_{1}^{\varepsilon}$ follows from the implicit function theorem. The case of $\gamma_{2}^{\varepsilon}$ is similar. Since $\gamma_{1}$ is include in $\{ \xi_{1} = \sqrt{E_{0}} \} \cap \{ x_{2} = 0 \}$ near $0$, $\gamma_{1}$ is an heteroclinic trajectory of $\widetilde{q}$ and then $\gamma_{1}^{\varepsilon} = \gamma_{1}$ for $\varepsilon$ small enough. We now prove by a contradiction argument that
\begin{equation} \label{a32}
\gamma_{2}^{\varepsilon} \cap ( \supp \psi \times \R^{2} ) \not\subset \{ x_{2} = 0 \}
\end{equation}
for $\varepsilon \neq 0$. If \eqref{a32} does not hold, the Hamiltonian equations for $\gamma_{2}^{\varepsilon} = ( x^{\varepsilon} ( t ) , \xi^{\varepsilon} ( t ) )$ gives
\begin{equation*}
\dot{x_{2}^{\varepsilon}} = 2 \xi_{2}^{\varepsilon} \qquad \text{ and } \qquad \dot{\xi_{2}^{\varepsilon}} = - \varepsilon ( \xi_{1}^{\varepsilon} - \sqrt{E_{0}} ) \psi ( x^{\varepsilon} ) .
\end{equation*}
Since $\xi_{1}^{0} = - \sqrt{E_{0}}$ when $\gamma_{2}^{0}$ passes above the support of $\psi$, we have $\dot{\xi_{2}^{\varepsilon}} \neq 0$ when $\gamma_{2}^{\varepsilon}$ passes above the support of $\psi$. Thus, $x_{2}^{\varepsilon}$ can not be constant, we get a contradiction with $x_{2}^{\varepsilon} = 0$ and \eqref{a32} holds true. Let us consider a function $\chi \in C^{\infty}_{0} ( \R^{2} ; [ 0 , 1 ] )$ with $\supp \chi \subset \supp \psi \setminus \{ x_{2} = 0 \}$ which does not vanish identically on $\pi_{x} ( \gamma_{2}^{\varepsilon} )$. We set
\begin{equation} \label{a40}
P = \widetilde{Q} - i \alpha ( h ) \chi .
\end{equation}
for some $h \vert \ln h \vert \leq \alpha ( h ) \leq 1$. By construction, $P$ is a Schr\"{o}dinger operator with a magnetic field and \eqref{a17} hold with $\CH = \gamma_{1}^{\varepsilon} = \gamma_{1}$ since the absorbing potential $\chi$ ``remove'' $\gamma_{2}^{\varepsilon}$ without ``touching'' $\gamma_{1}^{\varepsilon}$.

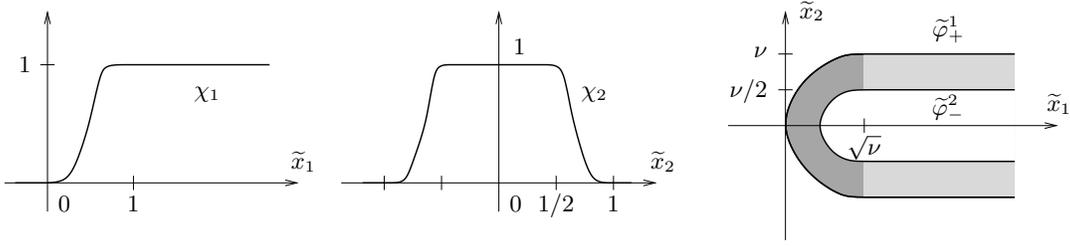
\begin{figure}
\begin{center}
\begin{picture}(0,0)%
\includegraphics{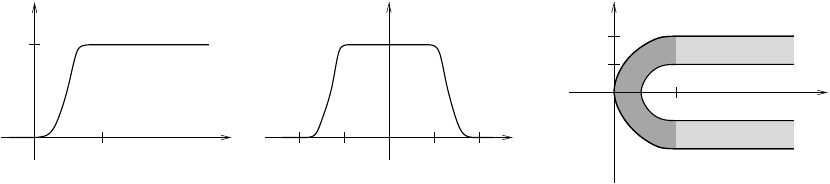}%
\end{picture}%
\setlength{\unitlength}{1184sp}%
\begingroup\makeatletter\ifx\SetFigFont\undefined%
\gdef\SetFigFont#1#2#3#4#5{%
  \reset@font\fontsize{#1}{#2pt}%
  \fontfamily{#3}\fontseries{#4}\fontshape{#5}%
  \selectfont}%
\fi\endgroup%
\begin{picture}(22094,4844)(-10371,-6383)
\put(226,-5761){\makebox(0,0)[lb]{\smash{{\SetFigFont{9}{10.8}{\rmdefault}{\mddefault}{\updefault}$0$}}}}
\put(1201,-5761){\makebox(0,0)[b]{\smash{{\SetFigFont{9}{10.8}{\rmdefault}{\mddefault}{\updefault}$1 / 2$}}}}
\put(2401,-5761){\makebox(0,0)[b]{\smash{{\SetFigFont{9}{10.8}{\rmdefault}{\mddefault}{\updefault}$1$}}}}
\put(301,-2461){\makebox(0,0)[lb]{\smash{{\SetFigFont{9}{10.8}{\rmdefault}{\mddefault}{\updefault}$1$}}}}
\put(5626,-2536){\makebox(0,0)[rb]{\smash{{\SetFigFont{9}{10.8}{\rmdefault}{\mddefault}{\updefault}$\nu$}}}}
\put(5626,-3361){\makebox(0,0)[rb]{\smash{{\SetFigFont{9}{10.8}{\rmdefault}{\mddefault}{\updefault}$\nu / 2$}}}}
\put(7651,-4561){\makebox(0,0)[b]{\smash{{\SetFigFont{9}{10.8}{\rmdefault}{\mddefault}{\updefault}$\sqrt{\nu}$}}}}
\put(9751,-3736){\makebox(0,0)[rb]{\smash{{\SetFigFont{9}{10.8}{\rmdefault}{\mddefault}{\updefault}$\widetilde{\varphi}_{-}^{2}$}}}}
\put(9751,-2086){\makebox(0,0)[rb]{\smash{{\SetFigFont{9}{10.8}{\rmdefault}{\mddefault}{\updefault}$\widetilde{\varphi}_{+}^{1}$}}}}
\put(3181,-4861){\makebox(0,0)[lb]{\smash{{\SetFigFont{9}{10.8}{\rmdefault}{\mddefault}{\updefault}$\widetilde{x}_{2}$}}}}
\put(1726,-3361){\makebox(0,0)[lb]{\smash{{\SetFigFont{9}{10.8}{\rmdefault}{\mddefault}{\updefault}$\chi_{2}$}}}}
\put(-10049,-2836){\makebox(0,0)[lb]{\smash{{\SetFigFont{9}{10.8}{\rmdefault}{\mddefault}{\updefault}$1$}}}}
\put(-7649,-5761){\makebox(0,0)[b]{\smash{{\SetFigFont{9}{10.8}{\rmdefault}{\mddefault}{\updefault}$1$}}}}
\put(-9224,-5761){\makebox(0,0)[lb]{\smash{{\SetFigFont{9}{10.8}{\rmdefault}{\mddefault}{\updefault}$0$}}}}
\put(-4349,-4861){\makebox(0,0)[lb]{\smash{{\SetFigFont{9}{10.8}{\rmdefault}{\mddefault}{\updefault}$\widetilde{x}_{1}$}}}}
\put(-6374,-3361){\makebox(0,0)[lb]{\smash{{\SetFigFont{9}{10.8}{\rmdefault}{\mddefault}{\updefault}$\chi_{1}$}}}}
\put(11476,-3661){\makebox(0,0)[lb]{\smash{{\SetFigFont{9}{10.8}{\rmdefault}{\mddefault}{\updefault}$\widetilde{x}_{1}$}}}}
\put(6301,-1711){\makebox(0,0)[lb]{\smash{{\SetFigFont{9}{10.8}{\rmdefault}{\mddefault}{\updefault}$\widetilde{x}_{2}$}}}}
\end{picture}%
\end{center}
\caption{The cut-off functions $\chi_{\bullet}$ and the phase $\widetilde{\varphi}$.} \label{f5}
\end{figure}

\subsubsection{Heteroclinic set of dimension $2$}

At this point, the heteroclinic set $\CH$ consists of a single trajectory $\gamma_{1}$. We now explain how to have a heteroclinic set of higher dimension. For that, we will add a small potential to $P$ localized near $0$ the following way. We write the operator $P$ of \eqref{a41} or \eqref{a40} as
\begin{equation}
P = P_{\rm re} - i P_{\rm im} \qquad \text{ with } \qquad P_{\rm re} = \Op ( p_{\rm re} ) ,
\end{equation}
where $p_{\rm re} ( x , \xi ) = q ( x , \xi )$ or $\widetilde{q} ( x , \xi )$ given in \eqref{a31} and $P_{\rm im}$ is the absorbing part in \eqref{a41} or \eqref{a40} depending on the situation. The point $\rho_{0} = ( 0 , 0 , \sqrt{E_{0}} , 0 ) $ belongs to $\gamma_{1}$ and $p_{\rm re} ( x , \xi ) - E_{0}$ is real principal type in a vicinity of $\rho_{0}$. Then there exists a symplectomorphism $( \widetilde{x} , \widetilde{\xi} ) = \kappa ( x , \xi )$ with $\kappa ( \rho_{0} ) = 0$ such that
\begin{equation} \label{a33}
( p_{\rm re} - E_{0} ) \circ \kappa^{- 1} ( \widetilde{x} , \widetilde{\xi} ) = \widetilde{p}_{\rm re} ( \widetilde{x} , \widetilde{\xi} ) = 2 \sqrt{E_{0}} \, \widetilde{\xi}_{1} .
\end{equation}
Moreover, since $H_{p_{\rm re} - E_{0}} ( \rho_{0} ) = ( 2 \sqrt{E_{0}} , 0 , 0 , 0 )$, the proof of the Darboux theorem (see e.g. Theorem 9.4 of Grigis and Sj\"{o}strand \cite{GrSj94_01}) shows that we can assume that $d \kappa ( \rho_{0} ) = \Id$ (this will guarantee that the Lagrangian manifolds considered in the sequel will project nicely on the $x$-space (resp. $\widetilde{x}$-space) near $\rho_{0}$ (resp. $0$)). Let $\Lambda_{+}^{1}$ (resp. $\Lambda_{-}^{2}$) denote the outgoing (resp. incoming) stable manifold from $\rho_{1}$ (resp. $\rho_{2})$. We have already seen between \eqref{a31} and \eqref{a32} that these Lagrangian manifolds belong to $p_{\rm re}^{- 1} ( E_{0} )$, intersect (transversally) along $\gamma_{1}$ and project nicely on the $x$-space near $\rho_{0}$. Then, the manifolds
\begin{equation*}
\widetilde{\Lambda}_{+}^{1} = \kappa ( \Lambda_{+}^{1} ) \qquad \text{ and } \qquad \widetilde{\Lambda}_{-}^{2} = \kappa ( \Lambda_{-}^{2} ) ,
\end{equation*}
are Lagrangian manifolds in $\widetilde{p}_{\rm re}^{- 1} ( 0 )$ which intersect (transversally) along the curve $\widetilde{\gamma}_{1} = \{ ( 2 \sqrt{E_{0}} t , 0 , 0 , 0 ) ; \ t \in \R \}$ and project nicely on the $\widetilde{x}$-space near $0$. Let $\widetilde{\varphi}_{+}^{1}$ and $\widetilde{\varphi}_{-}^{2}$ be their generating functions vanishing at $0$ (and then along $\pi_{\widetilde{x}} \widetilde{\gamma}_{1}$). Since these manifolds belong to $\widetilde{p}_{\rm re}^{- 1} ( 0 )$, the functions $\widetilde{\varphi}_{+}^{1}$ and $\widetilde{\varphi}_{-}^{2}$ depend on $\widetilde{x}_{2}$ only. Let $\chi ( \widetilde{x} ) = \chi_{1} ( \widetilde{x}_{1} / \sqrt{\nu} ) \chi_{2} ( \widetilde{x}_{2} / \nu ) \in C^{\infty} ( \R^{2} ; [ 0 , 1 ] )$ where $\nu > 0$ small and $\chi_{\bullet} \in C^{\infty} ( \R ; [ 0 , 1 ] )$ as in Figure \ref{f5}. In particular, $\chi_{2}^{- 1} ( 1 ) = [ - 1 / 2 , 1 / 2 ]$. We then define
\begin{equation} \label{a42}
\widetilde{\varphi} = \widetilde{\varphi}_{+}^{1} + ( \widetilde{\varphi}_{-}^{2} - \widetilde{\varphi}_{+}^{1} ) \chi  \qquad \text{ and } \qquad \widetilde{\Lambda} = \{ ( \widetilde{x} , \nabla \widetilde{\varphi} ( \widetilde{x} ) ) ; \ \widetilde{x} \in \R^{2} \} ,
\end{equation}
near $0$. Note that $0 \in \widetilde{\gamma}_{1} \subset \widetilde{\Lambda}$. By construction, $\widetilde{\varphi} ( \widetilde{x} ) = \widetilde{\varphi}_{+}^{1} ( \widetilde{x} )$ for $\widetilde{x}_{1} \leq 0$ or $\vert \widetilde{x}_{2} \vert \geq \nu$. Moreover, $\widetilde{\varphi} ( \widetilde{x} )$ depends only on $\widetilde{x}_{2}$ for $\widetilde{x}_{1} \geq \sqrt{\nu}$. Thus,
\begin{equation*}
\widetilde{p}_{\rm re} ( \widetilde{x} , \nabla \widetilde{\varphi} ( \widetilde{x} ) ) = 2 \sqrt{E_{0}} \nu^{- 1 / 2} ( \widetilde{\varphi}_{-}^{2} - \widetilde{\varphi}_{+}^{1} ) ( \widetilde{x} ) \chi_{1}^{\prime} ( \widetilde{x}_{1} / \sqrt{\nu} ) \chi_{2} ( \widetilde{x}_{2} / \nu ) ,
\end{equation*}
and
\begin{equation} \label{a34}
\supp \widetilde{p}_{\rm re} ( \widetilde{x} , \nabla \widetilde{\varphi} ( \widetilde{x} ) ) \subset [ 0 , \sqrt{\nu} ] \times [ - \nu , \nu ] .
\end{equation}
Moreover, since $\widetilde{\Lambda}_{+}^{1}$ and $\widetilde{\Lambda}_{-}^{2} $ intersect (transversally) along $\widetilde{\gamma}_{1}$, we have $\widetilde{\varphi}_{-}^{2} ( \widetilde{x} ) - \widetilde{\varphi}_{+}^{1} ( \widetilde{x} ) = \CO ( \widetilde{x}_{2}^{2} )$. It implies that
\begin{equation} \label{a38}
\partial_{\widetilde{x}}^{\alpha} \widetilde{p}_{\rm re} ( \widetilde{x} , \nabla \widetilde{\varphi} ( \widetilde{x} ) ) = \CO ( \nu^{3 / 2 - \vert \alpha \vert} ) ,
\end{equation}
for all $\alpha \in \N^{2}$. Let $x ( \widetilde{x} ) = \pi_{x} \kappa^{- 1} ( \widetilde{x} , \nabla \widetilde{\varphi} ( \widetilde{x} ) )$. Since $d \kappa^{- 1} ( 0 ) = \Id$, $\nabla \widetilde{\varphi} ( 0 ) = 0$ and $\Hess \widetilde{\varphi} = \CO ( 1 )$, $x ( \widetilde{x} )$ is a local diffeomorphism from a neighborhood of $0$ independent of $\nu$ of inverse function $\widetilde{x} ( x )$ defined on a neighborhood of $0$ independent of $\nu$. Furthermore, using that $\partial_{\widetilde{x}}^{\alpha} \widetilde{\varphi} = \CO ( \nu^{( 2 - \vert \alpha \vert )_{-}} )$ for all $\alpha \in \N^{2}$, we deduce
\begin{equation} \label{a36}
\partial_{x}^{\alpha} \widetilde{x} ( x ) = \CO ( \nu^{( 1 - \vert \alpha \vert )_{-}} ) ,
\end{equation}
for all $\alpha \in \N^{2}$. We set $\Lambda = \kappa^{- 1} ( \widetilde{\Lambda} )$. From the previous discussion, this Lagrangian manifold projects nicely on the $x$-space near $\rho_{0}$ and then $\Lambda = \{ ( x , \nabla \varphi ( x ) ) ; \ x \in \R^{2} \}$ near $\rho_{0}$ for some $\varphi \in C^{\infty} ( \R^{2} )$. The potential $W$ is defined by
\begin{equation} \label{a35}
W ( x ) = E_{0} - p_{\rm re} ( x , \nabla \varphi ( x ) ) .
\end{equation}
From \eqref{a33} and \eqref{a34}, $W ( x ) = - \widetilde{p}_{\rm re} ( \widetilde{x} ( x ) , \nabla \widetilde{\varphi} ( \widetilde{x} ( x ) ) ) \in C^{\infty}_{0} ( \R^{2} )$ is supported in a small neighborhood of $0$ (for $\nu > 0$ small). Moreover, combining \eqref{a38} and \eqref{a36}, we get
\begin{equation} \label{a37}
\partial_{x}^{\alpha} W ( x ) = \CO ( \nu^{3 / 2 - \vert \alpha \vert} ) ,
\end{equation}
for all $\alpha \in \N^{2}$.

Eventually, we set
\begin{equation}
\widehat{P} = \Op ( \widehat{p}_{\rm re} ) - i P_{\rm im} \qquad \text{ with } \qquad \widehat{p}_{\rm re} ( x , \xi ) = p_{\rm re} ( x , \xi ) + W ( x ) .
\end{equation}
Let us now describe the trapped set $K_{\widehat{P}} ( E_{0} )$. We first get an a priori estimate on the Hamiltonian flows.  Let $\rho ( t )$ (resp. $\widehat{\rho} ( t )$) be the Hamiltonian curve of $H_{p_{\rm re}}$ (resp. $H_{\widehat{p}_{\rm re}} = H_{p_{\rm re}} + H_{W}$) starting from $\rho ( 0 ) = \widehat{\rho} ( 0 )$ and denote $\gamma ( t ) = \rho ( t ) - \widehat{\rho} ( t )$. Since $\partial_{x} W = \CO ( \sqrt{\nu} )$ from \eqref{a37} and $\Hess p_{\rm re}$ is bounded, the Hamiltonian equations give
\begin{equation*}
\vert \dot{\gamma} \vert \leq C \vert \gamma \vert + C \sqrt{\nu} ,
\end{equation*}
for some $C > 0$. Then, the Gr\"{o}nwall lemma implies
\begin{equation} \label{a39}
\vert \rho ( t ) - \widehat{\rho} ( t ) \vert \leq \sqrt{\nu} e^{2 C t } ,
\end{equation}
for all $t \geq 0$. Using that  $p_{\rm re} = \widehat{p}_{\rm re}$ near $\rho_{1} , \rho_{2}$ since $W$ is localized only near $0$, using \eqref{a39} which shows than $W$ (and $\varepsilon ( \xi_{1} - \sqrt{E_{0}} ) x_{2} \psi ( x )$ for \eqref{a40}) hardly deviates the Hamiltonian trajectories and using Proposition B.12 of \cite{BoFuRaZe18_01} which guarantees than a trapped trajectory in $\widehat{p}_{\rm re}^{- 1} ( E_{0} )$ which touches the support of $V ( x_{1} \pm L , x_{2} )$ must be in its incoming/outgoing stable manifold, one can show that $K_{\widehat{P}} ( E_{0} )$ consists of the hyperbolic points $\rho_{0} , \rho_{1}$ and heteroclinic trajectories between them. Moreover, these heteroclinic trajectories are close to those for $p_{\rm re}$. In particular, those that go from $\rho_{2}$ to $\rho_{1}$ still pass through the absorbing potential and are remove from the trapped set of $\widehat{P}$ thanks to $- i P_{\rm im}$. Summing up,
\begin{equation*}
K_{\widehat{P}} ( E_{0} ) = \{ \rho_{1} , \rho_{2} \} \cup \CH ,
\end{equation*}
where $\CH$ is the union of the heteroclinic trajectories for $H_{\widehat{p}_{\rm re}}$ from $\rho_{1}$ to $\rho_{2}$. Let us now describe the set $\CH$. For that, let $\widehat{\Lambda}_{+}^{1}$ (resp. $\widehat{\Lambda}_{-}^{2}$) denote the outgoing (resp. incoming) stable manifold for $\widehat{p}_{\rm re}$ from $\rho_{1}$ (resp. $\rho_{2})$. On the left (with respect to $x_{1}$) of the support of $W$, we have $\widehat{\Lambda}_{+}^{1} = \Lambda_{+} ^{1} = \Lambda$ and on the right $\widehat{\Lambda}_{-}^{2} = \Lambda_{-}^{2}$. On the other hand, \eqref{a35} gives $\Lambda \subset \widehat{p}_{\rm re}^{- 1} ( E_{0} )$. Then, the Lagrangian manifold $\Lambda$ is union of integral curves of $H_{\widehat{p}_{\rm re}}$ (see e.g. the discussion above Proposition 1.4 of Dimassi and Sj\"{o}strand \cite{DiSj99_01}). Since this is the same for $\widehat{\Lambda}_{+}^{1}$, we get that $\widehat{\Lambda}_{+}^{1} = \Lambda$ on the right of the support of $W$. Summing up,
\begin{equation*}
\CH = \Lambda \cap \Lambda_{-}^{2} ,
\end{equation*}
on the right of the support of $W$. Applying the symplectomorphism $\kappa$ and noting $\widetilde{\CH} = \kappa ( \CH )$, the previous relation becomes
\begin{equation*}
\widetilde{\CH} = \widetilde{\Lambda} \cap \widetilde{\Lambda}_{-}^{2} = \big\{ ( \widetilde{x}_ {1} , \widetilde{x}_ {2} , 0 , \partial_{\widetilde{x}_{2}} \widetilde{\varphi} ( \widetilde{x}_ {2} ) ) ; \ \partial_{\widetilde{x}_{2}} \widetilde{\varphi} ( \widetilde{x}_ {2} ) = \partial_{\widetilde{x}_{2}} \widetilde{\varphi}_{-}^{2} ( \widetilde{x}_{2} ) \big\} ,
\end{equation*}
near $0$ for $\widetilde{x}_{1} \geq \sqrt{\nu}$. Since the intersection between $\widetilde{\Lambda}_{+}^{1}$ and $\widetilde{\Lambda}_{-}^{2}$ is transversal along $\widetilde{\gamma}_{1}$ and $\widetilde{\varphi}_{+}^{1} ( 0 ) = \widetilde{\varphi}_{-}^{2} ( 0 )$, we have $\widetilde{\varphi}_{-}^{2} ( \widetilde{x}_ {2} ) - \widetilde{\varphi}_{+}^{1} ( \widetilde{x}_ {2} ) \sim \alpha \widetilde{x}_ {2}^{2}$ and $\partial_{\widetilde{x}_{2}}  \widetilde{\varphi}_{-}^{2} ( \widetilde{x}_ {2} ) - \partial_{\widetilde{x}_{2}} \widetilde{\varphi}_{+}^{1} ( \widetilde{x}_ {2} ) \sim 2 \alpha \widetilde{x}_ {2}$ at the point $\widetilde{x}_ {2} = 0$ for some $\alpha \neq 0$. We assume in the sequel that $\alpha > 0$, the other case being similar. A direct computation and \eqref{a42} give
\begin{equation} \label{a43}
\partial_{\widetilde{x}_{2}} \widetilde{\varphi} - \partial_{\widetilde{x}_{2}} \widetilde{\varphi}_{-}^{2} = \big( \partial_{\widetilde{x}_{2}} \widetilde{\varphi}_{-}^{2} - \partial_{\widetilde{x}_{2}} \widetilde{\varphi}_{+}^{1} \big) ( \chi_{2} ( \widetilde{x}_{2} / \nu ) - 1 ) + ( \widetilde{\varphi}_{-}^{2} - \widetilde{\varphi}_{+}^{1} ) \partial_{\widetilde{x}_{2}} ( \chi_{2} ( \widetilde{x}_{2} / \nu ) ) .
\end{equation}
Since $\alpha > 0$ and $\chi_{2}$ is as in Figure \ref{f5}, the right hand side of \eqref{a43} is non-positive (resp. non-negative) for $\widetilde{x}_{2} \geq 0$ (resp. $\widetilde{x}_{2} \leq 0$). Moreover, this term vanishes iff $\vert \widetilde{x}_{2} \vert\leq 1 / 2$. Thus,
\begin{equation} \label{a44}
\widetilde{\CH} = \big\{ ( \widetilde{x}_ {1} , \widetilde{x}_ {2} , 0 , \partial_{\widetilde{x}_{2}} \widetilde{\varphi} ( \widetilde{x}_ {2} ) ) ; \ \widetilde{x}_{2} \in [ - 1 / 2 , 1 / 2 ] \big\} ,
\end{equation}
near $0$ for $\widetilde{x}_{1} \geq \sqrt{\nu}$. Coming back to the original variables, we deduce that $\CH$ has  Hausdorff, Minkowski and packing dimension $2$.

\subsubsection{Heteroclinic set of arbitrary dimension}

Up to now, we have constructed operators $P$ and $\widehat{P}$ whose trapped set at energy $E_{0}$ has dimension $1$ and $2$ respectively. It remains to explain how to change the previous construction to have a trapped set of fractal dimension $d$ for any given $d \in ] 1 , 2 [$. Modifying the Cantor construction,  one can obtain a compact set $K \subset [ - 1 / 4 , 1 / 4 ]$ of Hausdorff and Minkowski dimension $d - 1$ (see page 77 of Mattali \cite{Ma95_01}). The various fractal dimensions satisfy
\begin{equation} \label{a47}
\dim_{\text{Haus}} F \leq \underline{\dim}_{\text{pack}} F \leq \underline{\dim}_{\text{Mink}} F \quad \text{ and } \quad \underline{\dim}_{\text{pack}} F \leq \overline{\dim}_{\text{pack}} F \leq \overline{\dim}_{\text{Mink}} F ,
\end{equation}
for all $F \subset \R^{2 n}$ (see page 82 of \cite{Ma95_01}). Then, $K$ has also packing dimension $d - 1$. Up to a translation, we can assume that $0 \in K \subset [ - 1 / 2 , 1 / 2 ]$. Then, there exists a function $g \in C^{\infty}_{0} ( \R )$ such that $g^{- 1} ( 0 ) \cap [ - 1 , 1 ] = K$. We denote
\begin{equation*}
G ( t ) = \int_{0}^{t} s g^{2} ( s ) \, d s ,
\end{equation*}
$\widetilde{g}_{\bullet} ( \widetilde{x} ) = g ( \widetilde{x}_{2} / \nu )$ and $\widetilde{G} ( \widetilde{x} ) = G ( \widetilde{x}_{2} / \nu )$. We replace the phase $\widetilde{\varphi}$ of \eqref{a42} by
\begin{equation*}
\widetilde{\varphi}  = \widetilde{\varphi}_{+}^{1} + \big( \widetilde{\varphi}_{-}^{2} - \widetilde{\varphi}_{+}^{1} - e^{- 1 / \nu} \widetilde{G} \big) \chi ,
\end{equation*}
and construct $\widetilde{\Lambda} , \Lambda , W , \widehat{P} , \ldots$ as in the previous paragraph. The factor $e^{- 1 / \nu}$ and $G \in C^{\infty} ( \R )$ guarantee that the same results hold for $\nu > 0$ small enough excepted the formula \eqref{a44} for the heteroclinic set. Indeed, \eqref{a43} is replaced by
\begin{align}
\partial_{\widetilde{x}_{2}} \widetilde{\varphi} - \partial_{\widetilde{x}_{2}} \widetilde{\varphi}_{-}^{2} ={}& \big( \partial_{\widetilde{x}_{2}} \widetilde{\varphi}_{-}^{2} - \partial_{\widetilde{x}_{2}} \widetilde{\varphi}_{+}^{1} \big) ( \chi_{2} ( \widetilde{x}_{2} / \nu ) - 1 ) - e^{- 1 / \nu} \nu^{- 2} \widetilde{x}_{2} \widetilde{g}^{2} \chi_{2} ( \widetilde{x}_{2} / \nu ) \nonumber \\
&+ \big( \widetilde{\varphi}_{-}^{2} - \widetilde{\varphi}_{+}^{1} - e^{- 1 / \nu} \widetilde{G} \big) \partial_{\widetilde{x}_{2}} ( \chi_{2} ( \widetilde{x}_{2} / \nu ) ) ,  \label{a45}
\end{align}
in the present setting. We have $\vert e^{- 1 / \nu} \widetilde{G} \vert \leq \vert \widetilde{\varphi}_{-}^{2} - \widetilde{\varphi}_{+}^{1} \vert / 2$ on the support of $\partial_{\widetilde{x}_{2}} \chi_{2} ( \widetilde{x}_{2} / \nu )$. Since $\alpha > 0$, the three terms in the right hand side of \eqref{a45} are non-positive (resp. non-negative) for $\widetilde{x}_{2} \geq 0$ (resp. $\widetilde{x}_{2} \leq 0$). Moreover, their sum vanishes iff $\widetilde{x}_{2} \in [ - 1 / 2 , 1 / 2 ]$ and $\widetilde{x}_{2} \widetilde{g}^{2} ( \widetilde{x}_{2} ) = 0$. It means
\begin{equation*}
\widetilde{\CH} = \big\{ ( \widetilde{x}_ {1} , \widetilde{x}_ {2} , 0 , \partial_{\widetilde{x}_{2}} \widetilde{\varphi} ( \widetilde{x}_ {2} ) ) ; \ \widetilde{x}_{2} \in \nu K \big\} ,
\end{equation*}
near $0$ for $\widetilde{x}_{1} \geq \sqrt{\nu}$ and $\widetilde{\CH}$ has dimension $d$. We recall that the Hausdorff and Minkowski dimensions are invariant by diffeormophism (see Propositions 2.5 and 3.3 of Falconer \cite{Fa14_01}). Combinning with \eqref{a47}, this implies that $\CH$ has also Hausdorff, Minkowski and packing dimension $d$. Indeed, this is clear by invariance outside neighborhoods of $\rho_{1}$ and $\rho_{2}$. Since $V$ is radial, the trapped set $K_{\widehat{P}} ( E_{0} )$ near $\rho_{1}$ takes the form
\begin{equation*}
\big\{ \big( r \theta , \sqrt{ E_{0} - V ( r )} \theta \big) ; \ r \in [ 0 , 1 ] \text{ and } \theta \in K_{1} \big\} ,
\end{equation*}
where $x = x^{1} + r \theta$ are the spherical coordinates centered at $x^{1}$ and $K_{1}$ is a compact set of dimension $d - 1$. This set has dimension $d$ and the same argument holds near $\rho_{2}$. This construction provides an example satisfying all the assumptions of Theorem \ref{a1}.

\Subsection{A differential operator of order $4$} \label{s2}

\begin{figure}
\begin{center}
\begin{picture}(0,0)%
\includegraphics{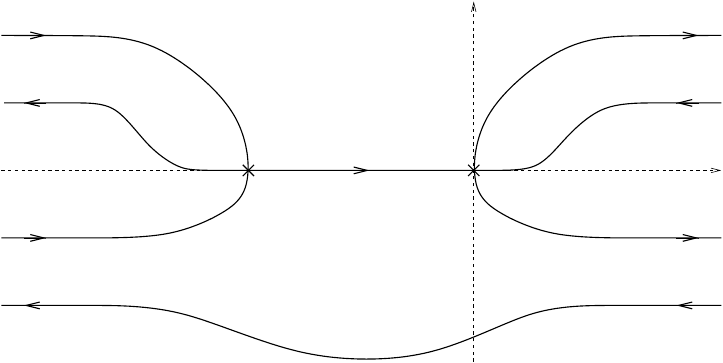}%
\end{picture}%
\setlength{\unitlength}{1184sp}%
\begingroup\makeatletter\ifx\SetFigFont\undefined%
\gdef\SetFigFont#1#2#3#4#5{%
  \reset@font\fontsize{#1}{#2pt}%
  \fontfamily{#3}\fontseries{#4}\fontshape{#5}%
  \selectfont}%
\fi\endgroup%
\begin{picture}(19266,9644)(-2432,-10283)
\put(6901,-4711){\makebox(0,0)[lb]{\smash{{\SetFigFont{9}{10.8}{\rmdefault}{\mddefault}{\updefault}$\CH^{0}$}}}}
\put(16426,-4711){\makebox(0,0)[lb]{\smash{{\SetFigFont{9}{10.8}{\rmdefault}{\mddefault}{\updefault}$x$}}}}
\put(9826,-5761){\makebox(0,0)[rb]{\smash{{\SetFigFont{9}{10.8}{\rmdefault}{\mddefault}{\updefault}$\rho_{2}^{0}$}}}}
\put(4426,-5761){\makebox(0,0)[lb]{\smash{{\SetFigFont{9}{10.8}{\rmdefault}{\mddefault}{\updefault}$\rho_{1}^{0}$}}}}
\put(10501,-1111){\makebox(0,0)[lb]{\smash{{\SetFigFont{9}{10.8}{\rmdefault}{\mddefault}{\updefault}$\xi$}}}}
\end{picture}%
\end{center}
\caption{The energy surface $q^{- 1} ( 0 ) \subset \R^{2}$.} \label{f1}
\end{figure}
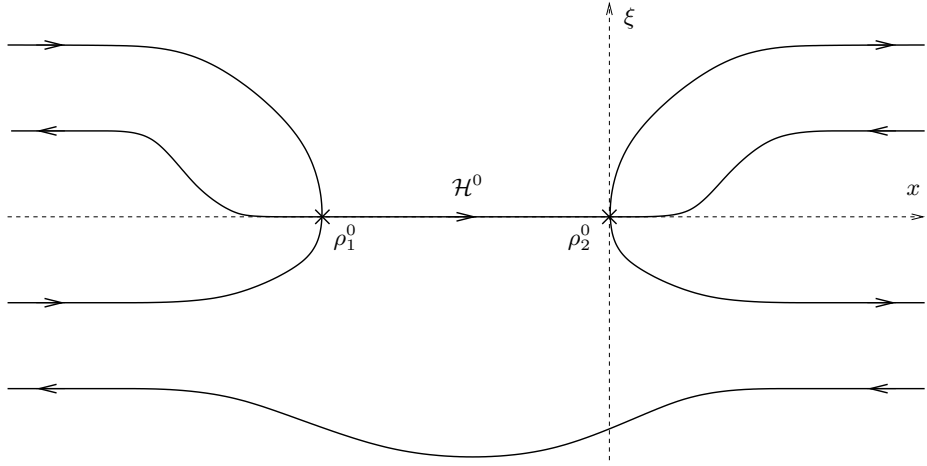

Since it is not possible to construct an example with a selfadjoint differential operator of order $2$, we exhibit in this part an example of selfadjoint differential operator of order $4$ in $\R^{2}$ satisfying the geometric structure \eqref{a17}. We start with a $1$-dimensional model:
\begin{equation*}
q ( x , \xi ) = \big( \xi^{2} - k ( x ) \big) \big( \xi^{2} - 4 k ( x ) \big) + 8 f ( x ) \xi^{3} - \frac{3}{2} x \Big( \frac{1}{4} + x \Big) f ( x ) \xi ,
\end{equation*}
where $f , 1 - k \in C^{\infty}_{0} ( \R ; [ 0 , 1])$ with $\one_{[ - 1 / 2 ,1/ 2 ]} \prec f \prec \one_{[ - 1 ,1 ]}$ and $\one_{[ - 1/2 ,1/ 2 ]} \prec 1 - k \prec \one_{[ - 3 , 3 ]}$. For $f , k$ well chosen, the energy surface $q^{- 1} ( 0 )$ is as in Figure \ref{f1}. In particular, the trapped set at energy $0$ is reduced to
\begin{equation} \label{a20}
K_{q} ( 0 ) = \{ \rho_{1}^{0} , \rho_{2}^{0} \} \cup \CH^{0},
\end{equation}
where $\rho_{1}^{0} = ( - 1 / 4 , 0 )$ and $\rho_{2}^{0} = ( 0 , 0 )$ are hyperbolic fixed points and $\CH^{0} = ] - 1 / 4 , 0 [ \times \{ 0 \}$ is a heteroclinic trajectory from $\rho_{1}^{0}$ to $\rho_{2}^{0}$. Note also that this symbol of order $4$ is elliptic and that $q ( x , \xi ) = ( \xi^{2} - 1 ) ( \xi^{2} - 4 ) + r ( x , \xi )$ where $r ( x , \xi )$ is compactly supported in $x$ and of order $3$ in $\xi$. Then, $q$ is a symbol satisfying \eqref{a17} in dimension $1$. To construct a setting with a heteroclinic set of higher dimension, we extend this example in $\R^{2}$ by adding a repulsive symbol in the other variable $y$. We set
\begin{equation} \label{a25}
p ( x , y , \xi , \eta ) = ( \xi^{2} - 1 ) ( \xi^{2} - 4 ) + \eta^{4} + \eta^{2} + r ( x , \xi ) \chi ( y ) + \lambda \varphi ( y ) ,
\end{equation}
with $\one_{\{ 0 \}} \prec \chi \in C^{\infty}_{0} ( \R ; [ 0 , 1 ] )$, $\varphi ( y ) = e^{- y^{2}} - 1$ and $\lambda > 1$. The term $\eta^{2}$ guarantees that the fixed points are hyperbolic and we add $\eta^{4}$ to define the resonances by the method of Helffer and Sj\"{o}strand \cite{HeSj86_01}.

\begin{lemma}\sl
For $\lambda$ large enough, the trapped set for $p$ at energy $0$ is given by
\begin{equation*}
K_{p} ( 0 ) = \{ \rho_{1} , \rho_{2} \} \cup \CH ,
\end{equation*}
where $\rho_{1} = ( - 1 / 4 , 0 , 0 , 0 )$ and $\rho_{2} = ( 0 , 0 , 0 , 0 )$ are hyperbolic fixed points and $\CH = ] - 1 / 4 , 0 [ \times \{ ( 0, 0 , 0 ) \}$ is a heteroclinic trajectory from $\rho_{1}$ to $\rho_{2}$.
\end{lemma}

Note that, since $p$ is globally elliptic, the Hamiltonian trajectories are complete, i.e. they are defined for all $t \in \R$.

\begin{proof}
Let $\rho ( t ) = ( x , y , \xi , \eta ) ( t )$ be a trapped Hamiltonian trajectory of $H_{p}$ with $p ( \rho ) = 0$. Then,
\begin{equation*}
\left\{ \begin{aligned}
\dot{y} &= 4 \eta^{3} + 2 \eta , \\
\dot{\eta} &= - r ( x , \xi ) \chi^{\prime} ( y ) + 2 \lambda y e^{- y^{2}} .
\end{aligned} \right.
\end{equation*}
In particular, we have $\dot{\wideparen{y^{2}}} = 2 ( 4 \eta^{3} + 2 \eta ) y$ and
\begin{equation} \label{a18}
\ddot{\wideparen{y^{2}}} = 2 ( 4 \eta^{3} + 2 \eta )^{2} + 2 ( 12 \eta^{2} + 2 ) \big( 2 \lambda y^{2} e^{- y^{2}} - r ( x , \xi ) y \chi^{\prime} ( y ) \big) .
\end{equation}
Assume $\supp \chi^{\prime} \subset [ - C , - 1 / C ] \cup [ 1 / C , C ]$ for some $C > 1$. On $p^{- 1} ( 0 )$, we have
\begin{equation*}
\xi^{4} + \eta^{4} -\lambda \lesssim \< \xi \>^{3} + \eta^{2} ,
\end{equation*}
and then $\vert \xi \vert \lesssim \lambda^{1/4}$ and $\vert \eta \vert \lesssim \lambda^{1/4}$. Thus,
\begin{equation*}
\lambda y^{2} e^{- y^{2}} - r ( x , \xi ) y \chi^{\prime} ( y ) \geq \lambda y^{2} e^{- y^{2}} - M \lambda^{3 / 4} \one_{ [ - C , - 1 / C ] \cup [ 1 / C , C ]} ( y ) \geq 0 ,
\end{equation*}
for some $M > 0$ and $\lambda$ large enough. As a consequence, \eqref{a18} gives
\begin{equation} \label{a19}
\ddot{\wideparen{y^{2}}} \geq 2 ( 4 \eta^{3} + 2 \eta )^{2} + 2 ( 12 \eta^{2} + 2 ) \lambda y^{2} e^{- y^{2}} \geq 0 ,
\end{equation}
which shows that $t \mapsto y^{2} ( t )$ is convex. Since this function is also bounded (because $\rho ( t )$ is trapped), we deduce that $y^{2} ( t )$ is constant. Therefore, \eqref{a19} implies $y ( t ) = \eta ( t ) = 0$ for all $t$. Since $p ( x, 0 , \xi , 0 ) = q ( x , \xi )$, $( x ( t ) , \xi ( t ) )$ is a trapped trajectory of $H_{q}$ of energy $0$ and the formula for $K_{p} ( 0 )$ follows from \eqref{a20}.

On the other hand,
\begin{equation*}
p ( x , y , \xi , \eta ) = \left\{ \begin{aligned}
&\frac{3}{8} \Big( x + \frac{1}{4} \Big) \xi + \eta^{2} - \lambda y^{2} + \CO \big( ( x , y , \xi , \eta )^{3} \big) &&\text{ near } \rho_{1} ,  \\
&- \frac{3}{8} x \xi + \eta^{2} - \lambda y^{2} + \CO \big( ( x , y , \xi , \eta )^{3} \big) &&\text{ near } \rho_{2} ,
\end{aligned} \right.
\end{equation*}
proving that $\rho_{1} , \rho_{2}$ are hyperbolic fixed points.
\end{proof}

At this point, we have constructed a symbol satisfying \eqref{a17} where $\CH$ consists of a single curve. To have a heteroclinic set $\CH$ of higher dimension, it is enough to modify slightly $p$ near $ \{ ( - 1 / 8 , 0, 0 , 0 ) \}$ as in the end of Section \ref{s3}. Then, the operator $\Op ( p )$ is a selfadjoint differential operator of order $4$ satisfying the geometric assumption of \eqref{a17}.

It remains to define the resonances for $\Op ( p )$ which does not have the form \eqref{a15}, and it seems difficult to use the classical analytic dilation/distortion method (see the discussion above Lemma \ref{a21}). Instead, we will follow the resonance theory of Helffer and Sj\"{o}strand \cite{HeSj86_01}, based on microlocal arguments. Using their notations, we take $r ( x , y ) = 1$, $R ( x, y ) = \< ( x , y ) \>$, $m_{0} ( x , y ) = 1$ and $N_{0} = 4$. Note that the function $( x , y ) \longmapsto e^{- y^{2}} -1$ is not bounded in the complex region \cite[(8.1)]{HeSj86_01}, but since it extends as a bounded holomorphic function in $\{ \vert \im y \vert \leq C \< \re y \> \}$, $C > 0$, the arguments of \cite{HeSj86_01} are still valid. It remains to construct an escape function at infinity. We note that $( \xi^{2} - 1 ) ( \xi^{2} - 4 ) = 0$ iff $\xi = \pm 1$ or $\xi = \pm 2$. The same way, $( \xi^{2} - 1 ) ( \xi^{2} - 4 ) - \lambda= 0$ iff
\begin{equation*}
\xi = \pm \xi_{\lambda} = \pm \sqrt{\frac{5 + \sqrt{9 + 4 \lambda}}{2}}
\end{equation*}
for $\lambda$ large enough. We consider $\chi_{\zeta} \in C^{\infty}_{0} ( \R ; [ 0 , 1] )$, with $\zeta = 1 , 2 , \xi_{\lambda}$, supported sufficiently close to $\{ - \zeta , \zeta \}$ and equal to $1$ near $\{ - \zeta , \zeta \}$. We then set
\begin{equation*}
G ( x , y , \xi , \eta ) = x \xi \big( \chi_{\xi_{\lambda}} ( \xi ) + \chi_{2} ( \xi ) - \chi_{1} ( \xi ) \big) + y \eta .
\end{equation*}
The choice of this escape function is justified as follows. It is natural to consider $G_{0} = x \xi + y \eta$, but the sign in front of $x \xi$ must be changed between $\xi = \pm 1$ and $\xi = \pm 2$ because the direction of propagation of the Hamiltonian curves at infinity changes between $\xi = \pm 1$ and $\xi = \pm 2$ (see Figure \ref{f1}). This explains why we can not use the standard analytic dilation/distortion method (and take $G_{0} = x \xi + y \eta$) to define the resonances. We have $G \in \dot{S}^{1,1} ( \R^{4} )$ and the lemma below shows that $G$ is an escape function at infinity. Then, $\Op ( p )$ satisfies the general assumptions of \cite[Chapter 8]{HeSj86_01} and the resonances can be defined as in this paper.

\begin{lemma}\sl \label{a21}
There exists a compact set $K \subset \R^{4}$ such that $H_{p} G \gtrsim 1$ on $p^{- 1} ( 0 ) \setminus K$.
\end{lemma}

\begin{proof}
Since we estimate $H_{p} G$ in $p^{- 1} ( 0 )$ outside a compact set $K \subset \R^{4}$, we work with
\begin{equation*}
p_{0} ( x , y , \xi , \eta ) = ( \xi^{2} - 1 ) ( \xi^{2} - 4 ) + \eta^{4} + \eta^{2} + \lambda \varphi ( y ) ,
\end{equation*}
instead of $p$. A direct computation gives
\begin{align}
H_{p_{0}} G &= \partial_{\xi} p_{0} \partial_{x} G - \partial_{x} p_{0} \partial_{\xi} G + \partial_{\eta} p_{0} \partial_{y} G - \partial_{y} p_{0} \partial_{\eta} G   \nonumber \\
&= ( 4 \xi^{4} - 10 \xi^{2} ) \big( \chi_{\xi_{\lambda}} ( \xi ) + \chi_{2} ( \xi ) - \chi_{1} ( \xi ) \big) + ( 4 \eta^{4} + 2 \eta^{2} ) + 2 \lambda y^{2} e^{- y^{2}} .  \label{a22}
\end{align}
For $\xi = \pm 1$, $\xi =  \pm 2$ and $\xi = \pm \xi_{\lambda}$, the function $4 \xi^{4} - 10 \xi^{2}$ is equal to $- 6$, $24$ and $9 + 4 \lambda + 5 \sqrt{9 + 4 \lambda}$ respectively. Thus, if the support of $\chi_{\zeta}$ has been chosen sufficiently close to $\{  - \zeta , \zeta\}$, the function $( 4 \xi^{4} - 10 \xi^{2} ) ( \chi_{\xi_{\lambda}}( \xi ) + \chi_{2} ( \xi ) - \chi_{1} ( \xi ) ) $ is non negative for $\xi \in \R$ and there exists $\nu > 0$ such that
\begin{equation} \label{a23}
( 4 \xi^{4} - 10 \xi^{2} ) \big( \chi_{\xi_{\lambda}}( \xi ) + \chi_{2} ( \xi ) - \chi_{1} ( \xi ) \big) \geq \nu ,
\end{equation}
if $\dist ( \xi , \{ \pm 1 , \pm 2 , \pm \xi_{\lambda} \} ) \leq \nu$.

For all $\delta > 0$, there exists $\varepsilon > 0$ such that
\begin{equation} \label{a24}
 ( 4 \eta^{4} + 2 \eta^{2} ) + 2 \lambda y^{2} e^{- y^{2}} \geq \varepsilon ,
\end{equation}
if $\eta^{2} \geq \delta$ or $y^{2} \in [ \delta , \delta^{- 1} ]$. On the other hand, if $\eta^{2} < \delta$ and $y^{2} < \delta$, we have $\vert \eta^{4} + \eta^{2} + \lambda \varphi ( y ) \vert \lesssim \delta$. Since $p_{0} ( x, y , \xi , \eta ) = 0$, this yields $\vert ( \xi^{2} - 1 ) ( \xi^{2} - 4 ) \vert \lesssim \delta$ and then $\dist ( \xi , \{ \pm 1 , \pm 2 \} ) \lesssim \delta$. For $\delta$ small enough, this region is covered by \eqref{a23}. The same way, if $\eta^{2} < \delta$ and $y^{2} > \delta^{- 1}$, we have $\vert \eta^{4} + \eta^{2} + \lambda \varphi ( y ) + \lambda \vert \lesssim \delta$. On $p_{0}^{- 1} ( 0 )$, it implies $\vert ( \xi^{2} - 1 ) ( \xi^{2} - 4 ) - \lambda \vert \lesssim \delta$ and then $\dist ( \xi , \{ \pm \xi_{\lambda} \} ) \vert \lesssim \delta$. Again, this region is covered by \eqref{a23} for $\delta$ small enough. Summing up, \eqref{a22}, together with the estimates \eqref{a23} and \eqref{a24}, shows that
\begin{equation*}
H_{p_{0}} G \geq \min ( \nu , \varepsilon ) ,
\end{equation*}
for all $( x , y , \xi , \eta ) \in p_{0}^{- 1} ( 0 )$.
\end{proof}

\Subsection{A matrix-valued Schr\"{o}dinger operator} \label{s5}

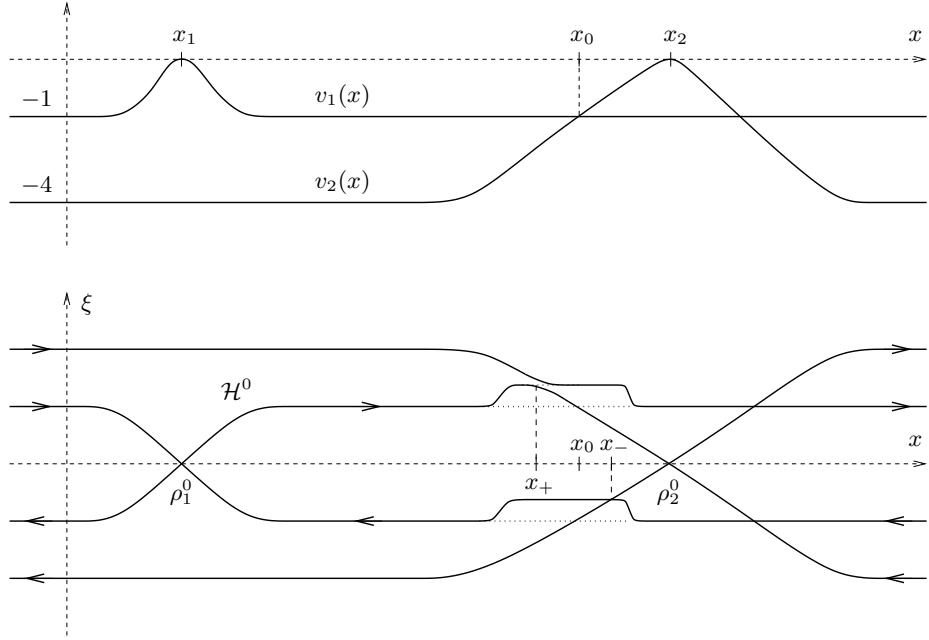
\begin{figure}
\begin{center}
\begin{picture}(0,0)%
\includegraphics{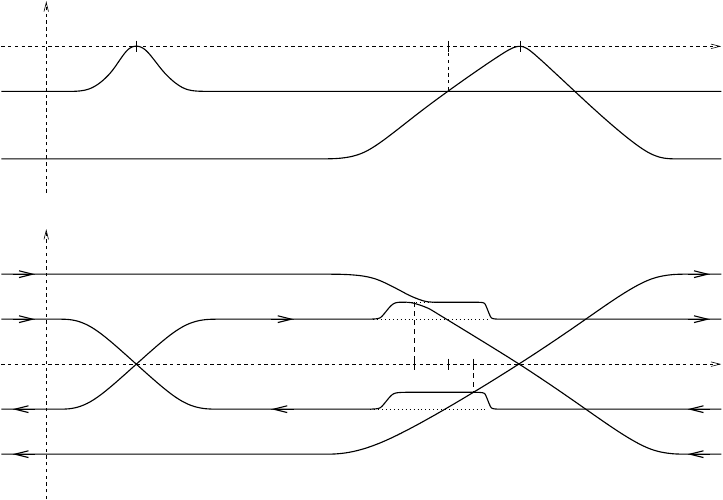}%
\end{picture}%
\setlength{\unitlength}{1184sp}%
\begingroup\makeatletter\ifx\SetFigFont\undefined%
\gdef\SetFigFont#1#2#3#4#5{%
  \reset@font\fontsize{#1}{#2pt}%
  \fontfamily{#3}\fontseries{#4}\fontshape{#5}%
  \selectfont}%
\fi\endgroup%
\begin{picture}(19266,13319)(-2432,-8783)
\put(1276,3689){\makebox(0,0)[b]{\smash{{\SetFigFont{9}{10.8}{\rmdefault}{\mddefault}{\updefault}$x_{1}$}}}}
\put(9601,3689){\makebox(0,0)[b]{\smash{{\SetFigFont{9}{10.8}{\rmdefault}{\mddefault}{\updefault}$x_{0}$}}}}
\put(11551,3689){\makebox(0,0)[b]{\smash{{\SetFigFont{9}{10.8}{\rmdefault}{\mddefault}{\updefault}$x_{2}$}}}}
\put(16426,3689){\makebox(0,0)[lb]{\smash{{\SetFigFont{9}{10.8}{\rmdefault}{\mddefault}{\updefault}$x$}}}}
\put(4576,614){\makebox(0,0)[b]{\smash{{\SetFigFont{9}{10.8}{\rmdefault}{\mddefault}{\updefault}$v_{2} ( x )$}}}}
\put(4576,2414){\makebox(0,0)[b]{\smash{{\SetFigFont{9}{10.8}{\rmdefault}{\mddefault}{\updefault}$v_{1} ( x )$}}}}
\put(-1499,2339){\makebox(0,0)[rb]{\smash{{\SetFigFont{9}{10.8}{\rmdefault}{\mddefault}{\updefault}$-1$}}}}
\put(-1499,539){\makebox(0,0)[rb]{\smash{{\SetFigFont{9}{10.8}{\rmdefault}{\mddefault}{\updefault}$- 4$}}}}
\put(1201,-5911){\makebox(0,0)[b]{\smash{{\SetFigFont{9}{10.8}{\rmdefault}{\mddefault}{\updefault}$\rho_{1}^{0}$}}}}
\put(11401,-5911){\makebox(0,0)[b]{\smash{{\SetFigFont{9}{10.8}{\rmdefault}{\mddefault}{\updefault}$\rho_{2}^{0}$}}}}
\put(-899,-1936){\makebox(0,0)[lb]{\smash{{\SetFigFont{9}{10.8}{\rmdefault}{\mddefault}{\updefault}$\xi$}}}}
\put(16426,-4786){\makebox(0,0)[lb]{\smash{{\SetFigFont{9}{10.8}{\rmdefault}{\mddefault}{\updefault}$x$}}}}
\put(8701,-5686){\makebox(0,0)[b]{\smash{{\SetFigFont{9}{10.8}{\rmdefault}{\mddefault}{\updefault}$x_{+}$}}}}
\put(10276,-4861){\makebox(0,0)[b]{\smash{{\SetFigFont{9}{10.8}{\rmdefault}{\mddefault}{\updefault}$x_{-}$}}}}
\put(9601,-4861){\makebox(0,0)[b]{\smash{{\SetFigFont{9}{10.8}{\rmdefault}{\mddefault}{\updefault}$x_{0}$}}}}
\put(2026,-3811){\makebox(0,0)[lb]{\smash{{\SetFigFont{9}{10.8}{\rmdefault}{\mddefault}{\updefault}$\CH^{0}$}}}}
\end{picture}%
\end{center}
\caption{The potentials $v_{1} , v_{2}$ on top and the energy surface $\det q = 0$.} \label{f2}
\end{figure}

In this part, we provide an example of a selfadjoint $2 \times 2$ matrix-valued Schr\"{o}dinger operator in $\R^{2}$ which satisfies the geometric structure \eqref{a17}. As in the previous section, we begin with a $1$-dimensional symbol. For $( x , \xi ) \in \R^{2}$, we set
\begin{equation*}
q ( x, \xi ) = \begin{pmatrix}
\xi^{2} + v_{1} ( x ) - \delta \chi ( x ) \xi & \varepsilon w ( x )  \\
\varepsilon w ( x ) & \xi^{2} + v_{2} ( x )
\end{pmatrix} ,
\end{equation*}
with $v_{1} , v_{2} , w , \chi \in C^{\infty} ( \R )$ and $0 < \varepsilon , \delta \ll 1$ as follows. The potentials $v_{j}$ are chosen as in Figure \ref{f2}. In particular, $v_{j}$ has a non-degenerate maximum at $x_{j}$ of value $0$, $v_{1} = - 1$ and $v_{2} = - 4$ outside a compact set and the two potentials are equal in $[ x_{1} , x_{2} ]$ only at the point $x_{0}$ in the vicinity of which $v_{1} = - 1$ and $v_{2}^{\prime} > 0$. We now add the small magnetic field $- \delta \chi ( x ) \xi$ to shift in opposite directions the two points in $\det q = 0$ with $x = x_{0}$. We take $\chi \in C^{\infty}_{0} ( \R ; [ 0 , 1 ] )$ supported and equal to $1$ near $x_{0}$. Then, the energy surface $\xi^{2} + v_{1} ( x ) - \delta \chi ( x ) \xi = 0$ near $x_{0}$ is given by
\begin{equation*}
\xi_{\pm}^{1} ( x ) = \frac{\delta \pm \sqrt{4 + \delta^{2}}}{2} = \pm1 + \delta / 2 + \CO ( \delta^{2} ) .
\end{equation*}
On the other hand, $\xi^{2} + v_{2} ( x ) = 0$ is equivalent to $\xi_{\pm}^{2} ( x ) = \pm \sqrt{- v_{2} ( x ) }$. Thus, for $\delta > 0$ small enough, $\xi_{+}^{1}$ and $\xi_{+}^{2}$ meet on $x_{+} < x_{0}$ whereas $\xi_{-}^{1}$ and $\xi_{-}^{2}$ meet on $x_{-} > x_{0}$. Eventually, $w \in C^{\infty}_{0} ( \R ; [ 0 , 1 ] )$ is supported and equal to $1$ near $x_{+}$ (in particular, $w = 0$ near $x_{-}$). Since the off-diagonal terms are equal to $\varepsilon$ near $x_{+}$, the crossing at $x_{+}$ becomes an avoided crossing and disappears.

Summing up, the trajectories in $\det q = 0$ are as in Figure \ref{f2}. The points $\rho_{1}^{0} = ( x_{1} , 0 )$ and $\rho_{2}^{0} = ( x_{2} , 0 )$ are hyperbolic fixed points. Near the 3 crossings (at $x_{-}$ and at the right of $x_{2}$), the matrix $q$ is diagonal. Then, these crossings play no role in the quantum dynamic (there is no possible exchange between the components) and cannot contribute to the trapped set. Thus, the trapped set at energy $0$ is given by
\begin{equation} \label{a30}
K_{q} ( 0 ) = \{ \rho_{1}^{0} , \rho_{2}^{0} \} \cup \CH^{0} ,
\end{equation}
where $\CH^{0}$ is the heteroclinic trajectory between $\rho_{1}^{0}$ and $\rho_{2}^{0}$.

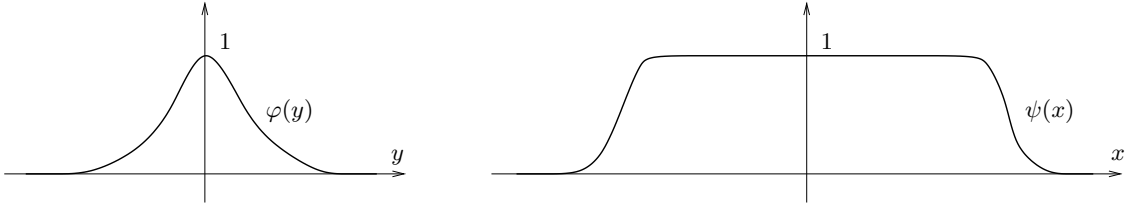
\begin{figure}
\begin{center}
\begin{picture}(0,0)%
\includegraphics{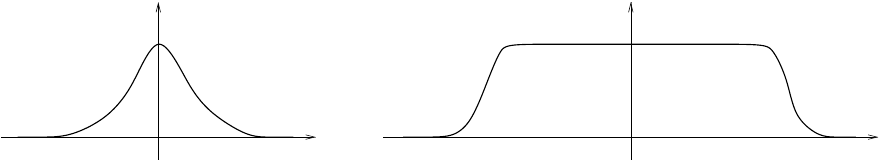}%
\end{picture}%
\setlength{\unitlength}{1184sp}%
\begingroup\makeatletter\ifx\SetFigFont\undefined%
\gdef\SetFigFont#1#2#3#4#5{%
  \reset@font\fontsize{#1}{#2pt}%
  \fontfamily{#3}\fontseries{#4}\fontshape{#5}%
  \selectfont}%
\fi\endgroup%
\begin{picture}(23444,4244)(-4221,-5783)
\put(18976,-4861){\makebox(0,0)[lb]{\smash{{\SetFigFont{9}{10.8}{\rmdefault}{\mddefault}{\updefault}$x$}}}}
\put(17176,-3961){\makebox(0,0)[lb]{\smash{{\SetFigFont{9}{10.8}{\rmdefault}{\mddefault}{\updefault}$\psi ( x )$}}}}
\put(1276,-3961){\makebox(0,0)[lb]{\smash{{\SetFigFont{9}{10.8}{\rmdefault}{\mddefault}{\updefault}$\varphi ( y )$}}}}
\put(3901,-4861){\makebox(0,0)[lb]{\smash{{\SetFigFont{9}{10.8}{\rmdefault}{\mddefault}{\updefault}$y$}}}}
\put(12901,-2536){\makebox(0,0)[lb]{\smash{{\SetFigFont{9}{10.8}{\rmdefault}{\mddefault}{\updefault}$1$}}}}
\put(301,-2536){\makebox(0,0)[lb]{\smash{{\SetFigFont{9}{10.8}{\rmdefault}{\mddefault}{\updefault}$1$}}}}
\end{picture}%
\end{center}
\caption{The functions $\varphi$ and $\psi$.} \label{f3}
\end{figure}

To have a heteroclinic set of higher dimension, we extend the symbol $q$ in $\R^{4}$ following the constructions of \eqref{a25}. This symbol can be written
\begin{equation*}
q ( x, \xi ) = \xi^{2} \Id_{2} + \begin{pmatrix}
- 1 & 0 \\
0 & - 4
\end{pmatrix} + R ( x , \xi ) ,
\end{equation*}
where $R ( x , \xi )$ is a smooth, compactly supported in $x$, affine in $\xi$, real-valued symmetric matrix. We set
\begin{equation}
p ( x , y , \xi , \eta ) = ( \xi^{2} + \eta^{2} ) \Id_{2} + \begin{pmatrix}
- 1 - \lambda & 0 \\
0 & - 4 - \lambda
\end{pmatrix} + R ( x , \xi ) \widetilde{\chi} ( y ) + \lambda \varphi ( y ) \psi ( x ) \Id_{2} ,
\end{equation}
with $\varphi , \psi \in C^{\infty}_{0} ( \R ; [ 0 , 1 ] )$ as in Figure \ref{f3}, $\one_{\{ 0 \}} \prec \widetilde{\chi} \in C^{\infty}_{0} ( \R ; [ 0 , 1 ] )$ and $\lambda > 1$. Moreover, we assume that $R ( \cdot , \xi ) \prec \psi$ for all $\xi \in \R$, $x \psi^{\prime} ( x ) \leq 0$, $\varphi ( y ) = 1 - y^{2}$ near $0$, $x \varphi^{\prime} ( x ) < 0$ for $x$ in the interior of $\supp \varphi \setminus \{ 0 \}$ and $\vert \varphi^{\prime} \vert > 0$ on the support of $\widetilde{\chi}^{\prime}$. In particular, $p$ is the symbol of a selfadjoint $2 \times 2$ matrix-valued Schr\"{o}dinger operator in $\R^{2}$ with smooth coefficients. Moreover,
\begin{equation*}
p ( x , y , \xi , \eta ) = ( \xi^{2} + \eta^{2} ) \Id_{2} + \begin{pmatrix}
- 1 - \lambda & 0 \\
0 & - 4 - \lambda
\end{pmatrix} ,
\end{equation*}
outside a compact set. On the other hand, for $y$ near $0$, we have
\begin{equation*}
p ( x , y , \xi , \eta ) = ( \xi^{2} + \eta^{2} ) \Id_{2} + \begin{pmatrix}
- 1 & 0 \\
0 & - 4
\end{pmatrix} + \lambda ( \psi ( x ) - 1 ) \Id_{2} + R ( x , \xi ) - \lambda y^{2} \psi ( x ) \Id_{2} ,
\end{equation*}
proving that $\rho_{1} = ( x_{1} , 0 , 0 , 0 )$ and $\rho_{2} = ( x_{2} , 0 , 0 , 0 )$ are hyperbolic fixed points and that the trapped set of $p$ at energy $0$ in $\{ ( y , \eta ) = ( 0 , 0 ) \}$ are those of $q$.

It remains to show that the trapped set of $p$ at energy $0$ satisfies
\begin{equation} \label{a29}
K_{p} ( 0 ) = \{ \rho_{1} , \rho_{2} \} \cup \CH ,
\end{equation}
where $\CH = ] \rho_{1} , \rho_{2} [$. The notion of trajectory for matrix-valued symbols can be subtle to defined in presence of eigenvalues crossings. To get around this difficulty, we use the approach of Jecko \cite{Je03_01} and construct a (scalar matrix-valued) escape function for $p$ outside $\{ y = 0 \}$. Since $p$ is elliptic, we have $\vert\xi \vert + \vert \eta \vert \leq C \< \lambda \>^{1 / 2}$ on the energy surface $\det p = 0$. Combining with the properties of the cut-off functions, it gives
\begin{align}
\{ p , y \eta \Id_{2} \} &= 2 \eta^{2} \Id_{2} - R ( x , \xi ) y \widetilde{\chi}^{\prime} ( y ) - \lambda y \varphi^{\prime} ( y ) \psi ( x ) \Id_{2}  \nonumber \\
&\geq \big( 2 \eta^{2} + 2 \vert y \varphi^{\prime} ( y ) \vert \psi ( x ) \big) \Id_{2} ,  \label{a26}
\end{align}
as $2 \times 2$ matrices for $\lambda$ fixed large enough in $\det p = 0$. Consider $g_{x}, g _{y} \in C^{\infty}_{0} ( \R ; [ 0 , 1 ] )$ such that $R ( \cdot , \xi ) \prec g_{x} \prec \psi$ for all $\xi \in \R$, $\widetilde{\chi} \prec g_{y} \prec \one_{\supp \varphi}$ and $x g_{x}^{\prime} ( x ) \leq 0$. Then, direct computations yield
\begin{align}
\{ p , x \xi ( 1 - g_{x} ( x ) g_{y} ( y ) ) \Id_{2} \} ={}& ( 2 \xi \Id_{2} + \partial_{\xi} R \widetilde{\chi} ( y ) ) \xi ( 1 - g_{x} g_{y} - x g_{x}^{\prime} g_{y} ) - 2 \eta x \xi g_{x} g_{y}^{\prime} \Id_{2}      \nonumber   \\
&- ( \partial_{x} R ( x , \xi ) \widetilde{\chi} ( y ) + \lambda \varphi ( y ) \psi^{\prime} ( x ) \Id_{2} ) x ( 1 - g_{x} ( x ) g_{y} ( y ) )   \nonumber   \\
={}& 2 \xi^{2}( 1 - g_{x} g_{y} - x g_{x}^{\prime} g_{y} ) \Id_{2} - 2 \eta x \xi g_{x} g_{y}^{\prime} \Id_{2}      \nonumber   \\
&- \lambda \varphi ( y ) x \psi^{\prime} ( x ) ( 1 - g_{x} ( x ) g_{y} ( y ) ) \Id_{2}   \nonumber   \\
\geq{}& 2 \xi^{2}( 1 - g_{x} g_{y} ) \Id_{2} - 2 \eta x \xi g_{x} g_{y}^{\prime} \Id_{2}   \nonumber   \\
\geq{}& 2 \xi^{2} ( 1 - g_{x} g _{y} ) \Id_{2} - C \vert y \varphi^{\prime} ( y ) \vert \psi ( x ) \Id_{2} ,   \label{a27}
\end{align}
for some $C > 0$. Here, we use the support properties of the cut-off functions, the bound on $\vert\xi \vert + \vert \eta \vert$, $x g_{x}^{\prime} ( x ) \leq 0$ and $x \psi^{\prime} ( x ) \leq 0$. Combining \eqref{a26} and \eqref{a27}, we get
\begin{equation} \label{a28}
\big\{ p , \big( y \eta + \delta x \xi ( 1 - g_{x} ( x ) g_{y} ( y ) ) \big) \Id_{2} \big\} \geq \big( 2 \eta^{2} + 2 \delta \xi^{2} ( 1 - g_{x} g _{y} ) + \vert y \varphi^{\prime} ( y ) \vert \psi ( x ) \big) \Id_{2} ,
\end{equation}
for $\delta > 0$ small enough. In $\det p = 0$, the right hand side of the previous equation is positive outside $\{ y = \eta = 0 \}$. Moreover, the trapped set of $p$ in $\{ y = \eta = 0 \}$ at energy $0$ is given by those of $q$. Thus, \eqref{a29} follows from \eqref{a30} and \eqref{a28}.

Summing up, the matrix-valued symbol $p$ satisfies \eqref{a17} where $\CH$ consists of a single curve. To have a heteroclinic set $\CH$ of higher dimension, one can modify slightly $p$ near $( \rho_{1} + \rho_{2} ) / 2$ as in Section \ref{s3} or Section \ref{s2}. Then, the operator $\Op ( p )$ is a selfadjoint matrix-valued Schr\"{o}dinger operator satisfying the geometric assumption of \eqref{a17}.

\bibliographystyle{amsplain}

\providecommand{\MRhref}[2]{%
  \href{http://www.ams.org/mathscinet-getitem?mr=#1}{#2}
}
\providecommand{\href}[2]{#2}

\end{document}